\documentclass[11pt]{amsart}
\usepackage{color}
\usepackage{amsmath}
\usepackage{amsxtra}
\usepackage{amscd}
\usepackage{amsthm}
\usepackage{amsfonts}
\usepackage{amssymb}
\usepackage{eucal}
\usepackage{epsfig}
\usepackage{graphics}
\usepackage[matrix,arrow]{xy}
\usepackage{hyperref}

\theoremstyle{plain}
\newtheorem{theorem}{Theorem}[section]

\newtheorem{thm}[theorem]{Theorem}
\newtheorem{cor}[theorem]{Corollary}
\newtheorem{prop}[theorem]{Proposition}
\newtheorem{lem}[theorem]{Lemma}
\newtheorem{conj}[theorem]{Conjecture}
\newtheorem{defi}[theorem]{Definition}
\newtheorem{example}[theorem]{Example}
\newtheorem{rem}[theorem]{Remark}

\newcommand{\Glie}{\mathfrak{g}}
\newcommand{\Yim}{\mathcal{Y}}
\newcommand{\Hlie}{\mathfrak{h}}

\newcommand{\ZZ}{\mathbb{Z}}
\newcommand{\CC}{\mathbb{C}}

\newcommand{\C}{\mathbb{C}}
\newcommand{\Q}{\mathbb{Q}}
\newcommand{\Z}{\mathbb{Z}}

\newcommand{\g}{\mathfrak{g}}
\newcommand{\bo}{\mathfrak{b}}
\newcommand{\tb}{\mathbf{\mathfrak{t}}}

\newcommand{\Psib}{\mbox{\boldmath$\Psi$}}
\newcommand{\Psibs}{\scalebox{.7}{\boldmath$\Psi$}}
\newcommand{\qbin}[2]{{\left[
\begin{matrix}{\,\displaystyle #1\,}\\
{\,\displaystyle #2\,}\end{matrix}
\right]
}}

\newcommand{\nc}{\newcommand}
\nc{\on}{\operatorname}
\nc{\la}{\lambda}
\nc{\wh}{\widehat}
\nc{\wt}{\widetilde}
\nc{\sw}{{\mathfrak s}{\mathfrak l}}
\nc{\ghat}{\wh{\g}}
\nc{\hhat}{\wh{\h}}
\nc{\mc}{\mathcal}
\nc{\bi}{\bibitem}
\nc{\pa}{\partial}
\nc{\ppart}{(\!(t)\!)}
\nc{\pparl}{(\!(\la)\!)}
\nc{\zpart}{(\!(z^{-1})\!)}
\nc{\n}{{\mathfrak n}}
\nc{\ol}{\overline}
\nc{\mb}{\mathbf}
\nc{\bb}{{\mathfrak b}}
\nc{\su}{\wh\sw_2}
\nc{\h}{{\mathfrak h}}
\nc{\can}{\on{can}}
\nc{\ntil}{\wt{\n}}
\nc{\pone}{{\mathbb P}^1}
\nc{\bs}{\backslash}
\nc{\al}{\alpha}
\nc{\gt}{{\mathfrak g}'}
\nc{\ds}{\displaystyle}

\nc{\OO}{{\mc O}}

\newcommand{\mfr}{\mathfrak{r}}

\theoremstyle{definition}

\begin{document}

\begin{title}[Extremal monomial property of $q$-characters and
    polynomiality of $X$-series] {Extremal monomial property of
    $q$-characters and polynomiality of the $X$-series}
\end{title}

\author{Edward Frenkel}

\address{Department of Mathematics, University of California,
  Berkeley, CA 94720, USA}

\author[David Hernandez]{David Hernandez}

\address{Universit\'e Paris Cit\'e and Sorbonne Universit\'e, CNRS, IMJ-PRG, F-75006, Paris, France}

\dedicatory{To Vyjayanthi Chari on her birthday}

\begin{abstract}
The character of every irreducible finite-dimensional representation
of a simple Lie algebra has the highest weight property. The
invariance of the character under the action of the Weyl group $W$
implies that there is a similar ``extremal weight property'' for every
weight obtained by applying an element of $W$ to the highest
weight. In this paper we conjecture an analogous ``extremal monomial
property'' of the $q$-characters of simple finite-dimensional modules
over the quantum affine algebras, using the braid group action on
$q$-characters defined by Chari. In the case of the identity element
of $W$, this is the highest monomial property of $q$-characters proved
in \cite{Fre2}. Here we prove it for simple reflections. Somewhat
surprisingly, the extremal monomial property for each $w \in W$ turns
out to be equivalent to polynomiality of the ``$X$-series''
corresponding to $w$, which we introduce in this paper. We show that
these $X$-series are equal to certain limits of the generalized Baxter
operators for all $w \in W$. Thus, we find a new bridge between
$q$-characters and the spectra of XXZ-type quantum integrable models
associated to quantum affine algebras. This leads us to conjecture
polynomiality of all generalized Baxter operators, extending the
results of \cite{FH}.
	\end{abstract}

\maketitle

\setcounter{tocdepth}{2}
\tableofcontents

\setcounter{tocdepth}{1}

\section{Introduction}

The theory of $q$-characters of finite-dimensional representations of
quantum affine algebras introduced in \cite{Fre} is a powerful tool to
investigate the structure of these representations and the
corresponding category, as well as the spectra of the Hamiltonians in
the quantum integrable models of XXZ type associated to these
representations. However, questions such as explicit
formulas for the $q$-characters of arbitrary simple modules still remain
open. In this paper we introduce a new property of
  $q$-characters that we believe can be useful for tackling these
  questions.

Let $U_q(\ghat)$ be the quantum affine algebra associated to a
finite-dimensional simple Lie algebra $\mathfrak{g}$.  Recall that the
$q$-character $\chi_q(V)$ of a finite-dimensional representation $V$
of $U_q(\ghat)$ is defined in \cite{Fre} as an element of the ring
$$\Yim := \Z[Y_{i,a}^{\pm 1}]_{i\in I, a\in\mathbb{C}^\times},$$ where
$I= \{1,\ldots, n\}$ and $n$ is the rank of $\g$. The variables
$Y_{i,a}$ are the analogues of the fundamental weights $\omega_i$ of
$\g$. There are also Laurent monomials $A_{i,a}$, $i\in I,
a\in\mathbb{C}^\times$, which are the analogues of the simple roots
$\al_i, i \in I$, of $\g$ in the following sense: Let $P$ be the
lattice of integral weights of $\g$ and $Z[P]$ its group
ring. Consider the homomorphism $\varpi: \Yim \to \Z[P]$ given by
$\varpi(Y_{i,a}) = \omega_i$ for all $i \in I, a \in \C^\times$. Then
$\varpi(A_{i,a}) = \alpha_i$.

Furthermore, for any finite-dimensional
$U_q(\ghat)$-module $V$, consider its restriction $\ol{V}$ to $U_q(\g)
\subset U_q(\ghat)$. Then the ordinary character $\chi(\ol{V})$ of
$\ol{V}$, viewed as an element of $\Z[P]$, is equal to
$\varpi(\chi_q(L))$, so the $q$-character of $V$ is a refinement of
its ordinary character.

The ordinary character $\chi(U)$ of any simple finite-dimensional
$U_q(\g)$-module $U$ (which coincides with the character of the
corresponding $\g$-module) has the familiar {\em highest weight
  property}: there is a dominant integral weight $\la \in P^+ \subset
P$, such that all other weights $\mu \in P$ in $\chi(U)$ have the form
$\mu = \la - \sum_{i \in I} a_i \al_i$, where $a_i \in \Z_{\geq 0}, i
\in I$.

The $q$-characters of simple $U_q(\ghat)$-modules have a similar
property, which was conjectured in \cite{Fre} and proved in
\cite{Fre2}. Namely, if $L$ is a simple $U_q(\ghat)$-module, then
there is a monomial $m$ in the variables $Y_{i,a}$ (raised to
non-negative powers only), called the {\em highest monomial}, such that the
$q$-character $\chi_q(L)$ of $L$ satisfies
\begin{equation}    \label{mainthm1}
\chi_q(L)\in m \cdot \mathbb{Z}[A_{i,a}^{-1}]_{i \in I,
  a\in\mathbb{C}^\times}.
\end{equation}
The highest monomial $m$ corresponds to the ordinary weight $\la :=
\varpi(m)$. Note that the construction of simple $U_q(\ghat)$-modules
given in \cite{Dri2,CP} immediately implies that the $\la$-weight
subspace in $\ol{L}$ is one-dimensional and so $m$ is the only
monomial in $\chi_q(L)$ with the ordinary
weight $\lambda$ (i.e. $\varpi(m) = \lambda$). Now, the
statement \eqref{mainthm1} (proved in \cite{Fre2}) means
that every other monomial in $\chi_q(L)$ has the form $m
\prod_{k=1,\ldots,p} A_{i_k,a_k}^{-1}$ for some $i_k \in I, a_k \in
\C^\times$. This statement has important consequences,
such as an algorithm \cite{Fre2} for computing the $q$-characters of
the fundamental representations of $U_q(\ghat)$.

Next, recall that the Weyl group $W$ of $\g$ acts on $P$. The
$W$-invariance of the character $\chi(U)$ of any simple
finite-dimensional $U_q(\g)$-module $U$ (which is equal to the
character of the corresponding simple $\g$-module) implies that
$\chi(U)$ has the following ``extremal weight property'' for each
element $w \in W$: all weights $\mu \in P$ appearing in $\chi(U)$ have
the form
$$
\mu = w(\la) - \sum_{i \in I} a_i w(\al_i),
$$
where $a_i \in \Z_{\geq 0}, i \in I$. Here $w(\la), w \in W$, is the
{\em extremal weight} of $U$ corresponding to $w$. Just like the
highest weight $\la$
of $U$, which corresponds to $w=e$, all extremal weights have
multiplicity one.

In this paper we describe a generalization of the extremal weight
property for simple finite-dimensional
$U_q(\ghat)$-modules.

To explain it, we use the braid group action on $\Yim$, which was
introduced by Chari in \cite{C}. This action is an analogue of the
$W$-action on $\Z[P]$. In fact, it is
intertwined with the $W$-action on $\Z[P]$ under the homomorphism
$\varpi$. Given $w \in W$, let $w = s_{i_1} \ldots s_{i_k}$ be its
reduced decomposition into simple
reflections $s_i, i \in I$. Then $T_w := T_{i_1} \ldots T_{i_k}$,
where $T_i, i \in I$, are the generators of the Chari braid group
action, is a well-defined automorphism of $\Yim$ assigned to
$w$. Moreover, $T_w$ preserves the group of monomials in $\Yim$.

Chari proved in \cite{C} that for every $w \in W$, the monomial
$T_w(m)$, where $m$ denotes the highest monomial of a simple
$U_q(\ghat)$-module $L$, appears in $\chi_q(L)$ and is in fact the
unique monomial $m_w$ in $\chi_q(L)$ with the ordinary
weight $\lambda$ (i.e. $\varpi(m_w) = w(\lambda)$).

Let $A_{i,a}^w := T_w(A_{i,a}), i \in I$ (these are the $q$-character
analogues of $w(\al_i)$). In this paper we conjecture the
  following {\em extremal monomial property} of the $q$-character of
  every simple module $L(m)$ for each $w \in W$:
\begin{equation}\label{maincom}
  \chi_q(L)\in T_w(m) \cdot \mathbb{Z}[(A_{i,a}^w)^{-1}]_{i \in I,
    a\in\mathbb{C}^\times}.
\end{equation}
This is the statement of Conjecture \ref{mainc}, one of the two main
conjectures of the present paper.

For $w = e$ (the identity element of $W$), formula \eqref{maincom}
becomes formula \eqref{mainthm1}, which has been proved in
\cite{Fre2}. Its validity implies formula \eqref{maincom} in the case
when $w$ is the longest element of $W$, as we show in Theorem
\ref{long}. And in Theorem \ref{refl} below, we will prove
\eqref{maincom} when $w$ is a simple reflection in the Weyl group.

We will also prove a weaker property for all $w \in W$ (in which we allow
negative powers of $A_{i,a}^w$ to appear):
\begin{equation}    \label{mainthm}
\chi_q(L)\in T_w(m) \cdot \mathbb{Z}[(A_{i,a}^w)^{\pm 1}]_{i \in I,
  a\in\mathbb{C}^\times}.
\end{equation}
This is Theorem \ref{pma}.

Somewhat surprisingly, it turns out that Conjecture \ref{mainc}
(expressed by formula \eqref{maincom}) is equivalent to Conjecture
\ref{twpol1} (which is our second main conjecture) about certain
explicit formal Taylor power series $X_{w(\omega_i)}(z)$ in the
variable $z$ which are defined in terms of the Drinfeld-Cartan
generators of $U_q(\ghat)$. We call them the $X$-{\em series}. Namely,
it is the statement that all eigenvalues of the (suitably
renormalized) $X$-series on any simple finite-dimensional
$U_q(\ghat)$-module (which could in principle be infinite formal power
series in $z$) are in fact polynomials in $z$.

In the case $w=e$, Conjecture \ref{twpol1} was proved in \cite{FH},
and in the case when $w$ is the longest element of $W$ it was proved
in \cite{Z} (this provides an alternative proof of Conjecture
\ref{mainc} in this case). In Section \ref{tmproof} below, we will
prove Conjecture \ref{twpol1} (and hence Conjecture
\eqref{maincom}) when $w$ is a simple reflection $s_i \in W, i \in I$.

\bigskip

What motivated us to study the $X$-series in the first place is the
fact that they appear as certain limits of the {\em generalized Baxter
  operators} which we defined in our recent paper \cite{FH4}. In that
paper, we introduced a family of simple modules
$L(\Psib_{w(\omega_i),a}), w\in W, i\in I$, in the category
$\mathcal{O}$ of modules over the Borel subalgebra of $U_q(\ghat)$
defined in \cite{HJ}. The generalized Baxter operators are
the transfer-matrices associated to these modules (up to some
normalization factors). In this paper, it will be more convenient to
consider a ``dual picture''; namely, we introduce the modules
$L'(\Omega(\Psib_{w(\omega_i),a})), w\in W, i\in I$, in the ``dual
category'' $\mathcal{O}^*$ (here $\Omega$ is an automorphism on
$\ell$-weights, see Sections \ref{dualcat} and \ref{nver}
for details). The corresponding transfer-matrices, which we denote by
$t_{w(\omega_i),a}(z,u)$, form the family of the ``dual'' generalized Baxter
operators.

In our earlier work \cite{FH} we considered the case $w =
e$, which corresponds to the transfer-matrices associated to the
modules $L'(\Psib_{\omega_i,a})$, $i
\in I$, $a\in\mathbb{C}^\times$, in $\OO^*$ (they are called the
prefundamental representations). We proved that all eigenvalues of
these transfer-matrices on any
simple finite-dimensional $U_q(\ghat)$-module $L$ are {\em
  polynomials}, up to an overall scalar factor depending only on
$L$. This crucial property has been used in \cite{FH} to describe
explicitly the spectra of the Hamiltonians of the quantum integrable
models of XXZ type associated to $U_q(\ghat)$ (following a conjecture
from \cite{Fre}). In the special case $\g=sl_2$, this reproduces the
celebrated results of Baxter on the spectra of the Heisenberg XXZ
model, which corresponds to $U_q(\wh{sl}_2)$ (see the Introduction of
\cite{FH} for details).

Motivated by this result, in this paper we conjecture that a similar
polynomiality property holds for all generalized Baxter operators
$t_{w(\omega_i),a}(z,u)$ (and not just the ones corresponding to
$w=e$, which was the case considered in \cite{FH}); that is, the
eigenvalues of $t_{w(\omega_i),a}(z,u)$ on simple finite-dimensional
$U_q(\ghat)$-modules are {\em polynomials} (up to a scalar factor),
for all $w \in W$ and $i \in I$ (Conjecture \ref{trm}). Using the dual
extended $TQ$-relations of Section \ref{nver}, we then obtain many
alternative descriptions of the spectra of the Hamiltonians of the
XXZ-type model associated to $U_q(\ghat)$ in terms of these
polynomials, one for every $w \in W$.

Next, we observe that the $X$-series $X_{w(\omega_i)}(z), w \in W, i
\in I$, can be identified with certain limits of the corresponding
generalized Baxter operators $t_{w(\omega_i),a}(z,u)$ (after replacing $z$
by $za$). Therefore polynomiality of the eigenvalues of the
Baxter operators implies the same property for the eigenvalues of the
$X$-series -- and this is precisely the statement of our Conjecture
\ref{twpol1}!

Thus, Conjecture \ref{twpol1} may be viewed as a special (degenerate)
case of the polynomiality conjecture for the generalized Baxter
operators. As we explained above, Conjecture \ref{twpol1} has a
special feature: it is equivalent to the extremal monomial property of
the $q$-characters, which is Conjecture \ref{mainc} discussed
above. This, of course, leads us to a natural question: Is there an
analogue of Conjecture \ref{mainc} for the generalized Baxter
operators?

It is also worth adding that the results and conjectures of this paper
are strongly inspired by the extended $TQ$- and $QQ$-relations we
introduced in \cite{FH4} as well as other tools from quantum
integrable models (of XXZ type) associated to quantum affine
algebras. These relations also inspired our construction in \cite{FH3}
of the Weyl group action on an extension of the ring $\Yim$ of
$q$-characters. We proved in \cite{FH4} that the subring of
$q$-characters in $\Yim$ is precisely the ring of invariants of $\Yim$
under this Weyl group action. However, the image of a monomial in
$\Yim$ under this action is in general an infinite linear combination
of monomials, and therefore the Weyl group does not act by permuting
monomials of a given $q$-character (in contrast to the Weyl group
action on the ordinary characters). That's why Conjecture
\ref{maincom} does not follow directly from this Weyl group action.

\bigskip

The paper is organized as follows.  In Section \ref{back} we recall
various definitions and results regarding quantum affine algebras and
their representations. We also discuss representations of the Borel
subalgebra of a quantum affine algebra from the category $\OO$ and the
dual category $\OO^*$. In Section \ref{braida} we recall the braid
group action defined by Chari \cite{C} and its generalization which we
defined in \cite{FH4}.  In Section \ref{nsqc2} we introduce twisted
root monomials parametrized by elements of the Weyl group, state the
extremal monomial property of $q$-characters (Conjecture \ref{mainc}),
and prove its weaker version (Theorem \ref{pma}).  In Section
\ref{epairing} we introduce formal power series
  $X_{w(\omega_i)}(z), w \in W, i \in I$, of the Drinfeld-Cartan
  generators of $U_q(\ghat)$, which we call the $X$-series. We also
introduce a pairing which we prove to be invariant with respect to an
action of the Weyl group $W$ (Proposition \ref{equith}). 
  These results enable us to relate Conjecture \ref{mainc} to the
  polynomiality of the eigenvalues of the $X$-series (Conjecture
  \ref{twpol1}) in Section \ref{tmproof}. Namely, we derive from the
  $W$-invariance of our pairing some identities between the
  eigenvalues of the $X$-series on a simple finite-dimensional
  $U_q(\ghat)$-module (Corollary \ref{inveigw}).  We then prove that
  up to a normalization factor (depending only on the module) these
  eigenvalues are expansions of rational functions (Theorem
  \ref{rattw}). This result is equivalent to Theorem
  \ref{pma}. Moreover, we conjecture that in fact these rational
  functions are polynomials (Conjecture \ref{twpol1}) and show that
  this property is equivalent to Conjecture \ref{mainc}. Finally, we
  use some additional results to prove Conjecture \ref{twpol1} (and
  hence Conjecture \ref{mainc}) for the simple reflections from the Weyl
  group (Theorem \ref{symsimp}). In Section \ref{trans}, we recall
the definition of the transfer-matrices, introduce the
  generalized Baxter operators, and the $TQ$-relations from
  \cite{FH4}, as well as their duals. We conjecture that
polynomiality of the eigenvalues of the $X$-series
(Conjecture \ref{twpol1}) extends to the (suitably renormalized)
generalized Baxter operators.

\bigskip

\noindent{\bf Acknowledgments.} It is a pleasure to
  dedicate this paper to Vyjayanthi Chari who has made fundamental
  contributions to representation theory of quantum affine
  algebras. In particular, the braid group action on the $q$-characters
  that she introduced in \cite{C} is one of the key ingredients of
  the present paper.

\section{Background on quantum affine algebras}    \label{back}

In this section, we recall some definitions and results on quantum affine
algebras and modules over them that we will need in the paper.
See \cite{CP} and the surveys \cite{CH,L} for more details. We
also recall some results on modules from the category $\OO$ of
the Borel subalgebra of a quantum
affine algebra (see \cite{HJ, FH} for more details) and the
corresponding quantum integrable models.

\subsection{Simple Lie algebras and Weyl groups}\label{Liealg}

Let $\Glie$ be a finite-dimensional simple Lie algebra of rank $n$ and
${\mathfrak{h}}$ its Cartan subalgebra. Using the
conventions of \cite{ka}, we denote by $C = (C_{i,j})_{i,j\in I}$ the
corresponding
Cartan matrix, where $I=\{1,\ldots, n\}$ and by $h^\vee$
(resp. $r^\vee$) the dual Coxeter number (resp. the lacing number) of
$\Glie$.

Let $\{\alpha_i\}_{i\in I}$,
$\{\alpha_i^\vee\}_{i\in I}$, $\{\omega_i\}_{i\in I}$,
$\{\omega_i^\vee\}_{i\in I}$ be the sets of simple roots, simple
coroots, fundamental weights, and fundamental coweights. Let
$$
Q:=\bigoplus_{i\in I}\Z\alpha_i, \qquad Q^+:=\bigoplus_{i\in
  I}\Z_{\ge0}\alpha_i, \qquad P:=\bigoplus_{i\in I}\Z\omega_i, \qquad
P^\vee:=\bigoplus_{i\in I}\Z\omega_i^\vee,
$$
and let $P_\Q := P\otimes \Q$.

The set of roots is denoted by
$\Delta\subset Q$ and $\Delta_+ = (\Delta\cap Q^+)$ is the set of
positive roots.  Let
$D=\mathrm{diag}(d_1,\ldots,d_n)$ be the unique diagonal matrix such
that $B=DC$ is symmetric and $d_i$'s are relatively prime positive
integers.  

We have a partial ordering on $P_\Q$ defined by the rule
$\omega\leq \omega'$ if and only if $\omega'-\omega\in Q^+$.

There is a unique such form $\kappa_0$ \cite{ka}
with the property that the dual form $(\kappa_0|_{\h})^{-1}$ on $\h^*$
to the restriction $\kappa_0|_{\h}$ of $\kappa_0$ to $\h \subset \g$
satisfies $(\alpha_i,\alpha_i)=2d_i$. We denote the restriction of
$(\kappa_0|_{\h})^{-1}$ to $P \subset \h^*$ by $(\cdot,\cdot)$. We have for $i,j\in I$,
$$(\alpha_i,\alpha_j) = d_i C_{i,j}\quad \text{ and } \quad
(\alpha_i,\omega_j) = d_j \delta_{i,j}.$$ 
The bases $(\alpha_i)_{i\in I}$,
$(\omega_i/d_i)_{i\in I}$ are dual to each other with respect to this
form.

Let $W$ be the Weyl group of $\mathfrak{g}$. It is generated by the simple
reflections $s_i, i\in I$. For $w\in W$, denote by $l(w)$ its
length. The formula
$$
s_k(\omega_j) = \omega_j - \delta_{k,j}\alpha_j, \qquad j,k \in I
$$
defines an action of the group $W$ on $P$.
The form $(\cdot,\cdot)$ is $W$-invariant, that is for any
$\lambda, \mu\in P$, we have
\begin{equation}\label{winv}(w \lambda,w \mu) = (\lambda,\mu)\text{
    for any $w\in W$.}\end{equation}

The following lemma is well-known (see \cite[Lemma 1.6, Corollary
    1.7, Proposition 1.15]{hu}).

\begin{lem}\label{wn} Let $w\in W$ and $i\in I$. We have the following.

(1) $w(\alpha_i)\in \Delta_+$ if and only if $l(ws_i) = l(w) + 1$.

(2) $w(\omega_i) = \omega_i$ if and only if $w$ is product of $s_j$, $j\neq i$.

(3) The $W$-orbits of the fundamental weights are disjoint. 
\end{lem}

Denote by $w_0$ the longest element of the Weyl group. We have the bar
involution on the set $I$ defined by the formula $w_0(\alpha_i) =
-\alpha_{\overline{i}}, i \in I$.

\subsection{Quantum affine algebras and their Borel
  subalgebras}    \label{bb}

Denote by $\wh{\Glie}$ the untwisted affine Kac-Moody Lie algebra
associated to $\Glie$. Let $(C_{i,j})_{0\leq i,j\leq n}$ be the
corresponding indecomposable Cartan matrix. We will use the Kac labels \cite{ka} denoted by
$a_0,\cdots,a_n$. We have $a_0 = 1$ and
we set $\alpha_0 = -(a_1\alpha_1 + a_2\alpha_2 + \cdots +
a_n\alpha_n)$.

Once and for all, fix a non-zero complex number $q$ that is not a
root of unity and 
  choose $h\in\CC$
such that $q = e^h$ and define $q^r \in \C^\times$ for any
$r\in\Q$ as $e^{hr}$. Since $q$ is not a root of unity, the map $\Q \to
  \C^\times$ sending $r \mapsto q^r$ is injective. Let $q^\Q := \{ e^{rh}, r \in \Q \}$ and 
	$\tb^\times := \on{Maps}(I,q^\Q) =
      \bigl(q^\Q\bigr)^I$ seen as a
      group with pointwise multiplication. We have a group isomorphism 
			$\overline{\phantom{u}}:P_\Q \longrightarrow
\tb^\times$ by setting
$$\overline{r \omega_i}(j):=q_i^{r \delta_{i,j}} = e^{r d_i h
  \delta_{i,j}}, \qquad i,j \in I, r \in \Q.
$$
The partial ordering on $P$ induces a partial ordering on
$\tb^\times$.

We will use the standard symbols for the $q$-integers
\begin{align*}
[m]_z=\frac{z^m-z^{-m}}{z-z^{-1}}, \quad
[m]_z!=\prod_{j=1}^m[j]_z,
 \quad 
\qbin{s}{r}_z
=\frac{[s]_z!}{[r]_z![s-r]_z!}. 
\end{align*}
We set $q_i=q^{d_i}$.

The quantum affine algebra $U_q(\ghat)$ of level 0 is the $\C$-algebra
with the generators $e_i,\ f_i,\ k_i^{\pm1}$ ($0\le i\le n$) and the
following relations:
\begin{align*}
&k_ik_j=k_jk_i,\quad k_0^{a_0}k_1^{a_1}\cdots k_n^{a_n}=1,\quad
&k_ie_jk_i^{-1}=q_i^{C_{i,j}}e_j,\quad k_if_jk_i^{-1}=q_i^{-C_{i,j}}f_j,
\\
&[e_i,f_j]
=\delta_{i,j}\frac{k_i-k_i^{-1}}{q_i-q_i^{-1}},
\\
&\sum_{r=0}^{1-C_{i.j}}(-1)^re_i^{(1-C_{i,j}-r)}e_j e_i^{(r)}=0\quad (i\neq j),
&\sum_{r=0}^{1-C_{i.j}}(-1)^rf_i^{(1-C_{i,j}-r)}f_j f_i^{(r)}=0\quad (i\neq j)\,.
\end{align*}
Here $0\le i,j\le n$ and $x_i^{(r)}=x_i^r/[r]_{q_i}!$ ($x_i=e_i,f_i$). 
The algebra $U_q(\wh{\Glie})$ has a Hopf algebra structure with
\begin{align*}
&\Delta(e_i)=e_i\otimes 1+k_i\otimes e_i,\quad
\Delta(f_i)=f_i\otimes k_i^{-1}+1\otimes f_i,
\quad
\Delta(k_i)=k_i\otimes k_i\,\text{ for $0\leq i\leq n$.}
\end{align*}

The algebra $U_q(\wh{\Glie})$ admits also a Drinfeld presentation with generators 
\begin{align}\label{dgen}
  x_{i,r}^{\pm}\ (i\in I, r\in\Z), \quad \phi_{i,\pm m}^\pm\ (i\in I,
  m\in\mathbb{Z}), \quad k_i^{\pm 1}\ (i\in I).
\end{align}
We will use the generating series $(i\in I)$
\begin{equation}\label{phrel}
  \phi_i^\pm(z) = \sum_{m\in \mathbb{Z}}\phi_{i, m}^\pm z^{ m} =
k_i^{\pm 1}\on{exp}\left(\pm (q_i - q_i^{-1})\sum_{r > 0} h_{i,\pm
  r} z^{\pm r} \right).\end{equation} 

The algebra $U_q(\wh{\Glie})$ has a $\ZZ$-grading defined by the formulas
$\on{deg}(e_i) = \on{deg}(f_i) = \on{deg}(k_i^{\pm 1}) = 0$ for $i\in
I$ and $\on{deg}(e_0) = - \on{deg}(f_0) = 1$.  

For any $a\in\CC^\times$, there is an automorphism $\tau_a :
U_q(\wh{\Glie})\rightarrow U_q(\wh{\Glie})$ such that $\tau_a(g) = a^m
g$ for any element $g$ of degree $m\in\ZZ$.

The {\em Borel subalgebra} $U_q(\wh\bb)$ is, by definition,
the subalgebra of $U_q(\wh{\Glie})$ generated by $e_i$ 
and $k_i^{\pm1}$ with $0\le i\le n$. It is a Hopf subalgebra of
$U_q(\wh{\Glie})$. For $a\in \C^\times$, the subalgebra $U_q(\wh\bb)$
is stable under $\tau_a$. We will denote its restriction to
$U_q(\wh\bb)$ by the same symbol.

\subsection{Category $\mathcal{O}$ of modules over the Borel
  subalgebra}    \label{catO}

For a $U_q(\wh\bb)$-module $V$ and $\omega\in \tb^\times$, we set the weight subspace of weight $\omega$:
\begin{align}
V_{\omega}:=\{v\in V \mid  k_i\, v = \omega(i) v\ (\forall i\in I)\}\,.
\label{wtsp}
\end{align}
We say that $V$ is Cartan-diagonalizable 
if $V=\underset{\omega\in \tb^\times}{\bigoplus}V_{\omega}$.

\begin{defi} A series $\Psib=(\Psi_{i, m})_{i\in I, m\geq 0}$ 
of complex numbers such that $\Psi_{i,0}
  \in q^\Q$ for all $i\in I$ is called an $\ell$-weight.

We denote by $\tb^\times_\ell$ the set of $\ell$-weights. 
\end{defi}

We will identify the collection $(\Psi_{i, m})_{m\geq 0}$ with its
generating series 
\begin{align*}
\Psib = (\Psi_i(z))_{i\in I},
\qquad
\Psi_i(z) := \underset{m\geq 0}{\sum} \Psi_{i,m} z^m.
\end{align*}

For example, for $i\in I$ and $a\in\mathbb{C}^\times$ we have the prefundamental $\ell$-weight $\Psib_{i,a}$ defined as 

\begin{align} 
(\Psib_{i,a})_j(z) = \begin{cases}
(1 - za)^{\pm 1} & (j=i)\,,\\
1 & (j\neq i)\,.\\
\end{cases} 
\end{align}

Since each $\Psi_i(z)$ is an invertible formal power series,
$\tb^\times_\ell$ has a natural group structure.  We also have a
surjective homomorphism of groups $\varpi :
  \tb^\times_\ell\rightarrow \tb^\times\simeq P_\Q$ defined by the assignment $(\Psib_i(z))_{i\in I}\mapsto(\Psib_i(0))_{i\in I} $.
  
\begin{defi} A $U_q(\wh\bb)$-module $V$ is called a
highest $\ell$-weight module with the highest $\ell$-weight
$\Psib\in \tb^\times_\ell$ if there is $v\in V$ such that 
$V =U_q(\wh\bb)v$ 
and the following equations hold:
\begin{align*}
e_i\, v=0\quad (i\in I)\,,
\qquad 
\phi_{i,m}^+v=\Psi_{i, m}v\quad (i\in I,\ m\ge 0)\,.
\end{align*}
\end{defi}

The highest $\ell$-weight $\Psib\in \tb^\times_\ell$ is uniquely
determined by $V$. The vector $v$ is said to be a highest
$\ell$-weight vector of $V$.

\begin{lem}\cite[Section 3.2]{HJ}\label{simple} 
For any $\Psib\in \tb^\times_\ell$, there exists a simple highest
$\ell$-weight module $L(\Psib)$ over $U_q(\wh\bb)$ with highest
$\ell$-weight $\Psib$. This module is unique up to isomorphism.
\end{lem}

\begin{defi}\cite{HJ}
For $i\in I$ and $a\in\CC^\times$, let $L_{i,a}^\pm = L(\Psib_{i,a}^{\pm 1})$.
We call $L_{i,a}^+$ (resp. $L_{i,a}^-$) a {\em positive (resp. negative)
prefundamental representation}.
\end{defi}


For $\lambda \in P_\Q$, we define the $1$-dimensional representation
$[\lambda] = L(\Psib_\lambda)$ with $\Psib_\lambda :=
(q_i^{\lambda(\alpha_i^\vee)})_{i\in I}$.


For $\lambda\in \tb^\times$, let $D(\lambda )=
\{\omega\in \tb^\times \mid \omega\leq\lambda\}$. Consider the category $\mathcal{O}$ introduced in \cite{HJ}.

\begin{defi}\label{defo} A $U_q(\wh\bb)$-module $V$ 
is said to be in category $\mathcal{O}$ if

(i) $V$ is Cartan-diagonalizable;

(ii) for all $\omega\in \tb^\times$ we have 
$\dim (V_{\omega})<\infty$;

(iii) there exist a finite number of elements 
$\lambda_1,\cdots,\lambda_s\in \tb^\times$ 
such that the weights of $V$ are in 
$\underset{j=1,\cdots, s}{\bigcup}D(\lambda_j)$.
\end{defi}

It is easy to check that the category $\mathcal{O}$ is a monoidal category.

Let $\mfr$ be the subgroup of $\tb^\times_\ell$ consisting of
$\Psib$ such that $\Psi_i(z)$ is an expansion in positive powers of
$z$ of a rational function in $z$ for any $i\in I$.

\begin{thm}\label{class}\cite{HJ} Let $\Psib\in\tb^\times_\ell$. 
The simple module $L(\Psib)$ is an object of $\mathcal{O}$ if and only
if $\Psib\in \mfr$.
\end{thm}


Let $K_0(\mathcal{O})$ be the (completed) Grothendieck ring of the
category $\OO$ as in \cite[Definition 2.15]{FH4} (see also
\cite[Section 3.2]{HL}).

\subsection{Monomials and finite-dimensional modules}\label{fdrep}

Following \cite{Fre}, we define the ring of Laurent polynomials $\Yim =
\ZZ[Y_{i,a}^{\pm 1}]_{i\in I,a\in\CC^\times}$ in the indeterminates
$\{Y_{i,a}\}_{i\in I, a\in \C^\times}$, containing the multiplicative
group $\mathcal{M}$  of Laurent monomials in $\Yim$.

For $i\in I, a\in\CC^\times$, this group contains, in particular,
the simple root monomials $A_{i,a}\in\mathcal{M}$ defined by the formula
\begin{multline}    \label{Ai}
A_{i,a} := \\ Y_{i,aq_i^{-1}}Y_{i,aq_i}
\Bigl(\prod_{\{j\in I|C_{j,i} = -1\}}Y_{j,a}
\prod_{\{j\in I|C_{j,i} = -2\}}Y_{j,aq^{-1}}Y_{j,aq}
\prod_{\{j\in I|C_{j,i} =
-3\}}Y_{j,aq^{-2}}Y_{j,a}Y_{j,aq^2}\Bigr)^{-1}\,.
\end{multline}

We have an injective group homomorphism 
$\mathcal{M}\rightarrow \mfr$ by 
$$Y_{i,a}\mapsto [\omega_i]\Psib_{i,aq_i^{-1}}\Psib_{i,aq_i}^{-1}.$$
Note that $\varpi(Y_{i,a}) = \omega_i$ and $\varpi(A_{i,a}) = \alpha_i$.

For any $m\in\mathcal{M}$, we will denote by $L(m)$ the corresponding simple $U_q(\wh\bb)$-module, 
where we identify the monomial $m\in\mathcal{M}$ with its image in
$\mfr$.

Let $\mathcal{C}$ be the category of (type $1$) finite-dimen\-sional
representations of $U_q(\wh{\Glie})$. 
 A monomial $m\in\mathcal{M}$ is said to be {\em
  dominant} if the powers of the variables $Y_{j,b}$, $j\in I$, 
$b\in\mathbb{C}^\times$, occurring in this monomial are all positive.
	We
  will denote by ${\mc M}^+$ the set of dominant monomials. 

Part (1) of the following theorem was proved in \cite{HJ} using the
result of part (2), which was established in \cite{CP}
following \cite{Dri2}.
 
\begin{thm}
(1) For $m \in \mathcal{M}$, the simple $U_q(\wh\bb)$-module $L(m)$ is
  finite-dimensional if and only if $m \in {\mc M}^+$. In this case, the
  action of $U_q(\wh\bb)$ on $L(m)$ can be uniquely extended to an
  action of $U_q(\wh{\Glie})$.

(2) Every simple module in the category $\mathcal{C}$ is of the form
  $L(m), m \in {\mc M}^+$.
\end{thm}



In particular, for $i\in I$, $a\in\CC^\times$, we have the fundamental representation $L(Y_{i,a})$ of
$U_q(\ghat)$. 


\subsection{The dual category $\mathcal{O}^*$}\label{dualcat}

For $V$ a Cartan-diagonalizable $U_q(\wh{\bo})$-module, 
we define a structure of $U_q(\wh{\bo})$-module on its
graded dual $V^* = \oplus_{\beta\in \tb^\times} V_\beta^*$ by
\begin{align*}
(x\,u)(v)=u\bigl(S^{-1}(x)v\bigr)\quad
(u\in  V^*, \ v\in V,\ x \in U_q(\mathfrak{b})).
\end{align*}

\begin{defi} Let $\mathcal{O}^*$ be the category of
  Cartan-diagonalizable $U_q(\bo)$-modules $V$ such that $V^*$ is in
  category $\mathcal{O}$.
\end{defi}

A $U_q(\wh{\mathfrak{b}})$-module $V$ is said to be of lowest $\ell$-weight 
$\Psib\in \tb^\times_\ell$ if there is $v\in V$ such that $V =U_q(\wh{\mathfrak{b}})v$ 
and the following hold:
\begin{align*}
U_q(\wh{\bo})^- v = \CC v\,,
\qquad 
\phi_{i,m}^+v=\Psi_{i, m}v\quad (i\in I,\ m\ge 0)\,.
\end{align*}
For $\Psib\in\tb^\times_\ell$, we have the simple $U_q(\wh{\bo})$-module
$L'(\Psib)$ of lowest $\ell$-weight $\Psi$.  The notions of
characters and $q$-characters of the objects of the category
$\mathcal{O}^*$ are defined in the same way as for the objects of the
category $\OO$.

\begin{prop}\label{dualweight}\cite{HJ} For $\Psib\in \tb^\times_\ell$ we
  have $(L'(\Psib))^* \simeq L(\Psib^{-1})$.
\end{prop}

In Section \ref{trans} we will discuss the prefundamental
  representations $R_{i,a}^\pm$ in
$\mathcal{O}^*$ defined by the property $(R_{i,a}^\pm)^* \simeq
L_{i,a}^\mp$.

\begin{rem}    \label{tau}
If $V$ is any finite-dimensional $U_q(\ghat)$-module, then
its double dual
$(V^*)^*$ is isomorphic to $V$ up to a twist by $\tau_{q^{-2r^\vee
    h^\vee}}$: $\tau_{q^{-2r^\vee
    h^\vee}}(V^*)^* \simeq V$. But for a general (simple) module $V$
in $\mathcal{O}^*$ or
  $\mathcal{O}$, it is not clear how to express $(V^*)^*$ in terms of
  $V$. This is closely related to the fact that the left and right
dual modules are not isomorphic to each other (not even up to a twist
by $\tau_{q^{-2r^\vee h^\vee}}$ in general).
\end{rem}

\subsection{$q$-characters}    \label{qchar}

For any module $V$ in the category $\mathcal{O}$, we define the character of
$V$ to be the following element of the group algebra of $P_\Q$ generated by the $[\omega] = \delta_{\omega,.}$ for $\omega\in P_\Q$. 

\begin{align}
\chi(V) = \sum_{\omega\in\tb^\times} 
\on{dim}(V_\omega) [\omega]\,.
\label{ch}
\end{align}

We recall the notion of $q$-character of a finite-dimensional
$U_q(\ghat)$-module.

For any $U_q(\wh\bb)$-module $V$ and
$\Psib\in\tb_\ell^\times$, the $\ell$-weight subspace of $V$ of $\ell$-weight $\Psib$ is
\begin{align}
V_{\Psibs} =
\{v\in V\mid
\exists p\geq 0, \forall i\in I, 
\forall m\geq 0,  
(\phi_{i,m}^+ - \Psi_{i,m})^pv = 0\}.
\label{l-wtsp} 
\end{align}


	
	We will identify ${\mc M}$ and
  $\Yim$ with their images in $\mfr$ and ${\mc E}_\ell$, respectively,
  under these homomorphisms.

It is proved in \cite{Fre} that for any finite-dimensional
$U_q(\wh{\Glie})$-module $V$, all $\ell$-weights appearing in $V$
are in the image of ${\mc M}$ in $\mfr$. Hence, one can define the $q$-character of $V$ 
\begin{align}
\chi_q(V) = 
\sum_{\Psibs\in\mfr}  
\mathrm{dim}(V_{\Psibs}) [\Psib]
\label{qch}
\end{align}
as an element of $\Yim = \mathbb{Z}[Y_{i,a}^{\pm 1}]_{i\in I, a\in\mathbb{C}^\times}$.

\begin{thm}\cite{Fre, Fre2}    \label{chiqprop}

(1) $\chi_q$ defines an injective ring morphism $K_0(\mathcal{C})\rightarrow \Yim$.

(3) For any simple finite-dimensional $U_q(\wh{\Glie})$-module $L(m)$, we have
$$ \chi_q(L(m))\in m \ZZ[A_{i,a}^{-1}]_{i\in I, a\in\CC^\times}.$$
\end{thm}

Let $\varphi:
  \mathcal{M}\rightarrow P$ be the restriction of the homomorphism
  $\varpi$ to ${\mc M}$. 
  For example, for $i\in I, a\in\CC^\times$, we have
  $\varphi(Y_{i,a}) = \omega_i$ and $\varphi(A_{i,a}) = \alpha_i$. If $m \in {\mc M}^+$, then $\varphi(m) \in P^+$.

 For any
finite-dimensional $U_q(\wh{\mathfrak{g}})$-module $V$, the ordinary
character of its restriction to $U_q(\mathfrak{g})$ is equal to
$\varphi(\chi_q(V))$.

\section{Chari's braid group action and its extension}\label{braida}

In this section, we recall the action of the braid group on $\Yim$
introduced by Chari in \cite{C} and its extension to $\Yim'$ which we
introduced in \cite[Section 3.2]{FH4}.

\subsection{Extremal monomials}\label{secextm}

 Consider the simple finite-dimensional
 $U_q(\wh{\mathfrak{g}})$-module $L(m), m \in {\mc M}^+$, and set
 $\la=\varphi(m) \in P^+$. The ordinary
 character of the restriction of $L(m)$ to $U_q(\g)$ is invariant
 under the action of the Weyl group $W$. For any $w\in W$, the dimension of the weight subspace
 of $L(m)$ of weight $w(\omega)$ is $1$, and it is spanned by an
 extremal weight vector (in the sense of \cite{kas}). We will denote this vector by
 $v_w$. (In particular, if $w=e$, the identity element of the Weyl group
 $W$, the vector $v_e$ is the highest weight vector of $L(m)$.) Hence $v_w$ is
 an $\ell$-weight vector.

 V. Chari introduced in \cite{C} a braid group action on
  $\Yim = \mathbb{Z}[Y_{i,a}^{\pm 1}]_{i\in I,
   a\in\mathbb{C}^\times}$. Namely, she defined operators $T_i, i \in
 I$, on $\Yim$ by formulas
\begin{equation}    \label{TY1}
  T_i(Y_{i,a}) = Y_{i,a}A_{i,aq_i}^{-1}\text{ and }T_i(Y_{j,a}) =
  Y_{j,a} \text{ if $j\neq i$}
\end{equation}
for all $i,j\in I $ and
$a\in\mathbb{C}^\times$. It was established in \cite{C} (a closely related result
  was proved earlier in \cite{BP}) that the $T_i$'s satisfy the
relations of the braid group corresponding to $\g$. 
  Therefore for any $w\in W$, we have a well-defined operator $T_w$
  acting on $\Yim$ obtained by considering a reduced decomposition $w = s_{i_1} \ldots s_{i_k}$ and setting
  $$
  T_w := T_{i_1} T_{i_2} \ldots T_{i_k}.$$

\begin{thm}\cite{C, CM}\label{occurlm} For each $w \in W$, the
  monomial corresponding to the $\ell$-weight of $v_w \in
  L(m)$ is $T_w(m)$. In particular, the multiplicity of
    $T_w(m)$ in $\chi_q(L(m))$ is equal to 1.
\end{thm}

 We call $T_w(m)$ the {\em extremal monomial} of
  $L(m)$ corresponding to $w$.

\medskip

We showed in \cite[Section 2.6]{FH4} that the following notation is
well-defined (that is, depends only on $w(\omega_i) \in P$, and not
separately on $w$ and $\omega_i$):
$$Y_{w(\omega_i),a} := T_w(Y_{i,a}).$$
For example, $Y_{\omega_i,a} = Y_{i,a}$, $Y_{s_i(\omega_i),a} = Y_{i,a} A_{i,aq_i}^{-1}$.

\subsection{Extension of Chari's action}    \label{genchari}

In \cite{FH4}, we extended Chari's braid group action to the following
extension $\mathcal{Y}'$ of the ring $\Yim$:
\begin{equation}    \label{Yprime}
\mathcal{Y}' :=
\mathbb{Z}[\Psib_{i,a}^{\pm 1}]_{i\in I,a\in\mathbb{C}^\times} \otimes_{\Z} \Z(P)
= \mathbb{Z}[\Psib_{i,a}^{\pm 1}, y_j^{\pm 1}]_{i\in I,a\in\mathbb{C}^\times;j \in I}
\supset \mathcal{Y},
\end{equation}
     where
      \begin{equation}    \label{YPsi}
        Y_{i,a} = [\omega_i]\Psib_{i,aq_i^{-1}}\Psib_{i,aq_i}^{-1}.
      \end{equation}
      
Recall the elements $\wt{\Psib}_{i,a} \in \Yim', i \in I$,
which we introduced in \cite{FH2}:
  \begin{multline}\label{psfh2}\wt{\Psib}_{i,a} :=
    \\ \Psib_{i,a}^{-1} \left(\prod_{j\in I,C_{i,j} =
      -1}\Psib_{j,aq_i}\right) \left(\prod_{j\in I, C_{i,j} =
  -2}\Psib_{j,a}\Psib_{j,aq^2}\right)\left(\prod_{j\in I, C_{i,j} =
      -3}\Psib_{j,aq^{-1}}\Psib_{j,aq}\Psib_{j,aq^3}\right).
  \end{multline}

Define a ring automorphism $\sigma$ of $\Yim'$ by the formula
\begin{equation}    \label{sigm}
\sigma(\Psib_{i,a}) := \Psib_{i,a^{-1}}^{-1}, \qquad i\in I, \quad
a\in\mathbb{C}^\times; \qquad \sigma([\omega]) = [\omega], \quad
\omega \in P.
\end{equation}
Define a ring automorphism $T'_i: \Yim' \to \Yim', i \in I$, by the
formulas
\begin{equation}    \label{TiPj}
T'_i(\Psib_{i,a}) = \sigma(\wt{\Psib}_{i,a^{-1}q_i^{-2}}^{-1}),
 \end{equation}
\begin{equation}
T'_i(\Psib_{j,a}) = \Psib_{j,a}, \quad j \neq i; \qquad
T'_i[\omega] = [s_i(\omega)] , \quad \omega \in P.
\end{equation}

\begin{thm}\cite{FH4}
(1)  The operators $T'_i, i \in I$, generate an action of the braid group
  associated to $\g$ on $\Yim'$.

(2) The restriction of $T'_i$ to $\Yim \subset
\Yim'$ is equal to the Chari operator $T_i$.
\end{thm}

Variations of the braid group action are also considered in \cite{GHL, W}.

\section{Extremal monomial property of
  $q$-characters}\label{nsqc2}

 In this section, we introduce $w$-twisted root monomials, where $w$
 is an arbitrary element of the Weyl group of $\g$. We then conjecture
 a new property of $q$-characters of simple finite-dimensional
 $U_q(\ghat)$-modules which we call the {\em extremal monomial
   property} corresponding to $w$ (Conjecture \ref{mainc}). We will
 prove a weaker property, stated in Theorem \ref{pma}. In the case of
 the identity element of $W$, Conjecture \ref{mainc} has been
 established in \cite{Fre2}. For the longest element of $W$, we have
 two proofs: the first is given in the proof of Theorem \ref{long} and
 the second is obtained by combining Theorem \ref{equi} and the proof
 of Conjecture \ref{twpol} obtained in \cite{Z}. In the case of simple
 reflections, we will prove Conjecture \ref{twpol1} (and hence
 Conjecture \ref{mainc} which is equivalent to it by Theorem
 \eqref{equi}) in Section \ref{tmproof} (Theorem \ref{symsimp}).

\subsection{Ordinary characters}

The character $\chi(V_i)$ of a fundamental representation $V_i$ of $\mathfrak{g}$ (or, equivalently, of $U_q(\mathfrak{g})$) 
is invariant for the action of the Weyl group $W$. Consequently 
$$\chi(V_i)\in [w(\omega_i)] + \sum_{\lambda \prec_{w} w(\omega_i)} \mathbb{Z} [\lambda],$$
where the partial ordering $\prec_w$ is defined by 
$$\lambda \preceq_w \mu \Leftrightarrow \mu - \lambda \in \sum_{j\in I} \mathbb{Z} [w(\alpha_j)]\Leftrightarrow w^{-1}(\lambda)\preceq w^{-1}(\mu).$$

The Weyl group symmetry of $q$-character established in \cite{FH3} does not give a direct analogue of this statement because the image of a 
monomial in $\mathcal{Y}$ is not a monomial $\mathcal{Y}$. Instead, we
find a new {\em extremal monomial property} of $q$-characters.

\subsection{Root monomials}

Let us introduce the $w$-{\em twisted root
    monomials} for $i\in I$,
$a\in\mathbb{C}^\times$ and $w\in W$:
$$A_{i,a}^w = T_w(A_{i,a})\text{ for $i\in I$ and $a\in\mathbb{C}^\times$.}$$
The corresponding weight is 
$$\varpi(A_{i,a}^w) = w(\alpha_i).$$


\begin{lem}\label{qroot} For any $w\in W$, $T_w$ induces a ring automorphism of $\mathbb{Z}[A_{i,a}^{\pm 1}]_{i\in I, a\in\mathbb{C}^\times}$.
\end{lem}

\begin{proof} It suffices to establish the statement for the $T_i, i
  \in I$.
It is established in \cite{CM} (by direct computation) that
$\mathbb{Z}[A_{j,a}^{\pm 1}]_{j\in I, a\in\mathbb{C}^\times}$ is
preserved by the $T_i$:

$$T_i(A_{i,a}) = A_{i,aq_i^2}^{-1}$$
$$T_i(A_{j,a}) = A_{j,a}\text{ if $C_{i,j} = 0$},$$
$$T_i(A_{j,a}) = A_{j,a}A_{i,aq_i}\text{ if $C_{i,j} = -1$},$$
$$T_i(A_{j,a}) = A_{j,a}A_{i,a}A_{i,aq^2}\text{ if $C_{i,j} = -2$},$$
$$T_i(A_{j,a}) = A_{j,a}A_{i,aq^{-1}}A_{i,aq}A_{i,aq^3}\text{ if $C_{i,j} = -3$}.$$
Now we have for $i,j\in I$:  
$$A_{i,a} = T_i(A_{i,aq_i^{-2}}^{-1})$$
$$A_{j,a} = T_i(A_{j,a})\text{ if $C_{i,j} = 0$},$$
$$A_{j,a} =  T_i(A_{j,a}A_{i,aq_i^{-1}})\text{ if $C_{i,j} = -1$},$$
$$A_{j,a} =  T_i(A_{j,a}A_{i,a}A_{i,aq^{-2}})\text{ if $C_{i,j} = -2$},$$
$$A_{j,a} =  T_i(A_{j,a}A_{i,aq}A_{i,aq^{-1}}A_{i,aq^{-3}})\text{ if $C_{i,j} = -3$}.$$
Hence $T_i$ induces an automorphism.
\end{proof}

\subsection{Extremal monomial property defined}

\begin{prop} There is a well-defined partial ordering $\preceq_w$ on
  $\mathcal{M}$ defined by the property
$$m\preceq_w m'\Leftrightarrow m' m^{-1} \in \mathbb{Z}[(A_{j,b}^w)^{-1}]_{j\in I,b\in\mathbb{C}^\times}\Leftrightarrow T_w^{-1}(m)\preceq T_w^{-1}(m').$$
\end{prop}

\begin{proof} This partial ordering is well-defined since the elements
  $A_{j,b}^w$ are algebraically independent (for $w=e$, this is known
    from \cite{Fre}, and for general $w \in W$ the result follows from Lemma
  \ref{qroot}).
\end{proof}

\begin{rem}
For $w = e$, this is sometimes called the Nakajima partial ordering.
\end{rem}

The following is the first main conjecture of this paper, which
expresses the {\em extremal monomial property} of the $q$-characters
of simple finite-dimensional $U_q(\ghat)$-modules. It is a
generalization to an arbitrary element $w \in W$ of the highest
monomial property of the $q$-characters, conjectured in
\cite{Fre} and proved in \cite{Fre2}, which corresponds to the identity
element $w = e$ of the Weyl group (see Theorem \ref{chiqprop},(ii)).

\begin{conj}\label{mainc} Let $L(m)$ be the simple
finite-dimensional $U_q(\wh{\mathfrak{g}})$-module with the
highest monomial $m$ and $w\in W$. Then
\begin{equation}    \label{taylor}
\chi_q(L(m)) \in T_w(m) \mathbb{Z}[(A_{j,b}^w)^{-1}]_{j\in
  I,b\in\mathbb{C}^\times}.
\end{equation}
Equivalently,
$$\chi_q(L(m)) = T_w(m) + \sum_{m'\prec_w T_w(m)} c_{m'}
  \; m', \qquad c_{m'} \in \Z_{\geq 0}.$$
\end{conj}

\begin{rem} Conjecture \ref{mainc} is equivalent to the
following statement: for all $i\in I$, $a\in\mathbb{C}^\times$, and
$w\in W$,
$$\chi_q( L(Y_{i,a})) \in Y_{w(\omega_i),a}
\mathbb{Z}[(A_{j,b}^w)^{-1}]_{j\in I,b\in\mathbb{C}^\times}.$$
\end{rem}

\begin{thm}    \label{long}
Conjecture \ref{mainc} holds when $w$ is the identity element and the
longest element of the Weyl group.
\end{thm}

\begin{proof}
In the case $w=e$, Conjecture \ref{mainc} has been proved in
\cite{Fre2}. Applying to this result the involution studied in
\cite[Section 4.2]{h3}, we obtain
$$
\chi_q(L(Y_{i,a})) \in Y_{w_0(\omega_i),a} \mathbb{Z}[A_{j,b}]_{j\in I,
  b\in\mathbb{C}^\times}.
$$
This implies the statement of Conjecture \ref{mainc} for $w = w_0$
(the longest element of the Weyl group) because one can readily see that
$(A_{j,b}^{w_0})^{-1} = A_{\overline{j},bq^{h^{\vee}r^\vee}}$,
	where the bar involution of $I$ is defined by the
formula $\al_{\overline{i}} = -w_0(\al_i)$. In fact, $T_{w_0}(Y_{j,b}) = Y_{\overline{j},bq^{h^{\vee}r^\vee}}^{-1}$ by \cite{C, Fre2}.
\end{proof}

The partial ordering $\preceq_w$ induces a partial ordering on simple
modules, which we denote by the same symbol:
$$L(m)\preceq_w L(m') \Leftrightarrow m\preceq_w m'.$$
Then the statement of the Conjecture \ref{mainc} implies that for
simple modules $L(m)$ and $L(m')$, we have the following relation in
the Grothendieck ring of finite-dimensional $U_q(\ghat)$-modules: 
$$[L(m)\otimes L(m')] = [L(mm')] + \sum_{ L \prec_w L(mm')}
 k_L [L], \qquad k_L \in \Z_{\geq 0}$$
because $T_w(mm') = T_w(m) T_w(m')$ occurs with multiplicity $1$ in
$\chi_q(L(mm'))$.

\begin{rem} It would be interesting to find out whether the partial
  orderings $\preceq_w$ on simple modules are related to the partial
  orderings introduced in \cite{kkop}.\end{rem}

\begin{example} (i) Let $\Glie = sl_3$ and $w = s_1$. Then 
$$Y_{w(\omega_1),a} = Y_{1,aq^2}^{-1}Y_{2,aq}\text{ , }Y_{w(\omega_2),a} = Y_{2,a}\text{ , }A_{1,a}^w = A_{1,aq^2}^{-1}\text{ , }A_{2,a}^w = A_{2,a}A_{1,aq},$$
$$\chi_q(L(Y_{1,a})) = Y_{w(\omega_1),a}(1 + (A_{1,aq^{-1}}^w)^{-1} + (A_{1,aq}^w)^{-1} (A_{2,aq^2}^w)^{-1}),$$
$$\chi_q(L(Y_{2,a})) = Y_{w(\omega_2),a}(1 + (A_{2,aq}^w)^{-1} + (A_{2,aq}^w)^{-1}(A_{1,a}^w)^{-1}).$$

(ii) Let $\Glie = B_2$ ($d_1 = 2$, $d_2 = 1$) and $w = s_1s_2$. Then 
$$Y_{w(\omega_1),a} = Y_{1,aq^4}^{-1}Y_{2,aq}Y_{2,aq^3}\text{ , }Y_{w(\omega_2),a} = Y_{1,aq^5}^{-1}Y_{2,aq^4},$$
$$A_{1,a}^w = A_{1,aq^2}A_{2,aq^2}A_{2,a}
\text{ , }A_{2,a}^w = A_{2,aq^2}^{-1}A_{1,aq^4}^{-1},$$
$$\chi_q(L(Y_{1,a})) = Y_{w(\omega_1),a}(1 
+ (A_{1,aq^2}^wA_{2,a}^w)^{-1}
+ (A_{1,aq^2}^wA_{2,a}^wA_{1,a}^wA_{2,aq^{-2}}^w)^{-1} 
+ (A_{1,aq^2}^w)^{-1}$$
$$+ (A_{2,aq^{-2}}^wA_{1,aq^{-2}}^wA_{2,aq^{-4}}^w)^{-1}),$$
$$\chi_q(L(Y_{2,a})) = Y_{w(\omega_2),a}(1 +
(A_{1,aq^3}^wA_{2,aq}^w)^{-1}
+ (A_{2,aq^{-1}}^wA_{1,aq^{-1}}^wA_{2,aq^{-3}}^w)^{-1}
+ (A_{2,aq^{-1}}^w)^{-1}).$$

\end{example}

\begin{rem} The example (i) shows that the extremal monomial property
  does not follow from a reduction to 
the $sl_2$-case (contrary to the case $w = e$, as shown
  in the proof of the highest monomial property in \cite{Fre2}). Indeed, 
the factors $(A_{1,aq^{-1}}^w)^{-1}$ and $(A_{1,aq}^w)^{-1}$ occur with different spectral parameters therein.
\end{rem}

 The following theorem is a weaker version of Conjecture
  \ref{mainc}; namely, in Conjecture \ref{mainc} we have polynomials
  in $(A_{j,b}^w)^{-1}$ (see formula \eqref{taylor}), whereas in this
  theorem we have polynomials in $(A_{j,b}^w)^{\pm 1}$ (see formula
  \eqref{laurent}).

\begin{thm}\label{pma}
 Let $L(m)$ be the simple finite-dimensional
$U_q(\wh{\mathfrak{g}})$-module with the highest monomial $m$ and $w
  \in W$. Then
\begin{equation}    \label{laurent}
\chi_q(L(m)) \in T_w(m) \mathbb{Z}[(A_{j,b}^w)^{\pm 1}]_{j\in
  I,b\in\mathbb{C}^\times}\subset
\mathbb{Z}[(Y_{w(\omega_j),b})^{\pm 1}]_{j\in
  I,b\in\mathbb{C}^\times}.
\end{equation}
\end{thm}

\begin{proof} It suffices to establish the result for $m = Y_{i,a}$, $i\in I$ and $a\in\mathbb{C}^\times$. 
By definition, $Y_{w(\omega_i),a}$ occurs in $\chi_q(L(Y_{i,a}))$ with
multiplicity $1$. This implies that
$$\chi_q(L(Y_{i,a})) \in Y_{w(\omega_i),a} \mathbb{Z}[(A_{j,b})^{\pm 1}]_{j\in I,b\in\mathbb{C}^\times}.$$
But, by Lemma \ref{qroot}, we have 
$$\mathbb{Z}[(A_{j,b})^{\pm 1}]_{j\in I,b\in\mathbb{C}^\times} = T_w(\mathbb{Z}[(A_{j,b})^{\pm 1}]_{j\in I,b\in\mathbb{C}^\times} ) = \mathbb{Z}[(A_{j,b}^w)^{\pm 1}]_{j\in I,b\in\mathbb{C}^\times} .$$
\end{proof}

\section{An invariant pairing}\label{epairing}

In this section, we develop technical tools to establish Conjecture
\ref{mainc} for simple reflections, and potentially for
  all elements of the Weyl group. We will introduce the
    {\em $X$-series} which are certain formal Taylor power
  series $X_{w(\omega_i)}(z), w \in W, i \in I$, of
  Drinfeld-Cartan generators in $U_q(\ghat)$. We also introduce a pairing
  that we prove to be invariant with respect to the braid group
  operators $T_w, w \in W$ (Proposition \ref{equith}). We then relate
the $\ell$-weights $\Psib_{w(\omega_i),a}$ to the $X$-series
$X_{w(\omega_i)}(za)$.
(Proposition \ref{tpsi}).

\subsection{Fundamental and prefundamental series}

 In \cite{FH}, we defined the {\em prefundamental series}
  $X_i(z), i\in I$, by the formulas
  \begin{equation}\label{fdser} X_i(z) =
\text{exp}\left( \sum_{m > 0} z^m
\frac{\tilde{h}_{i,-m}}{[d_i]_q[m]_{q_i}}\right),
\end{equation}
where
\begin{equation}    \label{wth}
\wt{h}_{i,-m} = \sum_{j\in I}[d_j]_q \wt{C}_{j,i}(q^m) h_{j,-m}\text{
  and }[d_i]_q h_{i,-m} = \sum_{j\in I} C_{j,i}(q^m) \wt{h}_{j,-m}.
\end{equation}

\begin{rem}
  In \cite{FH}, the series $X_i(z)$ was denoted by $T_i(z)$.
\end{rem}

  The series $X_i(z)$ are formal Taylor power series in $z$ which
  are related to
  the {\em fundamental series} $Y_i(z)$ defined in \cite{Fre} as follows:
$$Y_i(z) = \overline{k}_i^{-1} \frac{X_i(zq_i^{-1})}{X_i(zq_i)} = \overline{k}_i^{-1}\text{exp}\left( \sum_{m > 0} z^{m}
(q^{-1} - q)\wt{h}_{i,-m}\right),$$
where
the elements $\wt{h}_{i,n}$ are defined by formula \eqref{wth}, and
the elements $\overline{k}_i$ are defined in the adjoint version of quantum affine algebra by the relation:
$$\prod_{j\in I} \overline{k}_j^{C_{j,i}} = k_i.$$

The relation between $X_i(z)$, $Y_i(z)$, and the 
  eigenvalues of transfer-matrices will be discussed below (this will also justify the terminology). 
Note that $X_i(z)$ commutes with the $x_{j,r}^\pm$ for $j\neq i$ and $r\in\ZZ$ (since $[\wt{h}_{i,-m},x_{j,r}^\pm] = 0$ for $i\neq j$).

\begin{example} In the case $\g = sl_2$, we have $X_1(z) = \on{exp}\left(
    \sum_{m > 0} z^{m} \frac{h_{1,-m}}{[m]_q(q^m +
      q^{-m})}\right)$.\end{example}

 Since the series $Y_i(az), i \in I, a \in \C^\times$, are
  algebraically
independent in $U_q(\wh{\Hlie})[[z]])$,
there is a unique injective group homomorphism
\begin{equation}\label{identi}\mathcal{I} : \mathcal{M} \rightarrow (U_q(\wh{\Hlie})[[z]])^\times\end{equation}
  such that $\mathcal{I}(Y_{i,a}) = Y_i(za)$.

Recall from \cite{Fre} that
\begin{multline}\label{ay}\phi_i^-(z^{-1}) = \\ Y_i(zq_i)Y_i(zq_i^{-1})\prod_{j,C_{j,i} = -1}Y_j(z)\prod_{j,C_{j,i} = -2}Y_j(zq^{-1})Y_j(zq)\prod_{j,C_{j,i} = -3}Y_j(zq^{-2})Y_j(z)Y_j(zq^2).\end{multline}
 Comparing with formula \eqref{Ai}, we obtain that $\mathcal{I}(A_{i,a}) = \phi_i^-(z^{-1}a^{-1})$.

\subsection{Definition of the pairing}

Given $m\in{\mathcal M}$ and a formal power series $h(z)$ in
$U_q(\wh{\mathfrak{h}})[[z]]$, we define
a formal power series
$$\langle h(z),m \rangle\in \mathbb{C}[[z]],$$
as the eigenvalue of $h(z)$ on the $\ell$-weight space associated to
the monomial $m$.
This defines a pairing 
$$\langle \cdot,\cdot \rangle : U_q(\wh{\mathfrak{h}})[[z]] \times \mathcal{M} \rightarrow \mathbb{C}[[z]].$$
It is related to the inner product $(\cdot,\cdot)$ on the
lattice $P$ of weights of $\g$ introduced in Section \ref{Liealg} as
follows:
\begin{equation}    \label{pairi}
  \langle k_j,m \rangle = q^{(\alpha_j,\varpi(m))},\qquad j \in I.\
\end{equation}

This pairing satisfies the following property: for
$m_1,m_2\in\mathcal{M}$, $i\in I$, $m\neq 0$, we have
$$\langle h_{i,m},m_1m_2 \rangle = \langle h_{i,m},m_1 \rangle +
\langle h_{i,m},m_2 \rangle \text{ , } \langle k_i,m_1m_2 \rangle =
\langle k_i,m_1\rangle \langle k_i,m_2 \rangle,$$
\begin{equation}\label{mult}
  \langle X_i(z),m_1m_2 \rangle =
  \langle X_i(z),m_1 \rangle \langle X_i(z),m_2 \rangle \text{ , }
  \langle Y_i(z),m_1m_2 \rangle =
  \langle Y_i(z),m_1 \rangle \langle Y_i(z),m_2 \rangle.\end{equation}

By definition, for $i,j\in I$ and $m > 0$, we have 
$$\langle \wt{h}_{i,-m},Y_{j,a} \rangle = [d_i]_q [m]_{q_i}
a^{-m}\frac{\wt{C}_{i,j}(q^m)}{m},$$
and therefore
$$\langle h_{i,-m},Y_{j,a} \rangle = \delta_{i,j}
a^{-m}\frac{[m]_{q_i}}{m}\text{ , }\langle \wt{h}_{i,-m},A_{j,a}
\rangle = \delta_{i,j} a^{-m}[d_i]_q\frac{[m]_{q_i}}{m}.$$

According to the results of \cite{FH}, for $i,j\in I$,
$a\in\mathbb{C}^\times$, we have
\begin{equation}    \label{XY}
\langle X_i(z),Y_{j,a} \rangle = \exp \left(\sum_{r > 0}
        z^r a^{-r}
\frac{\wt{C}_{i,j}(q^r)}{r}\right),
\end{equation}
\begin{equation}    \label{XA}
  \langle X_i(z),A_{j,a} \rangle = (1 - za^{-1})^{-\delta_{i,j}}.
\end{equation}
In particular, if $m = \prod_{j\in I,
    b\in\CC^\times} Y_{j,b}^{u_{j,b}(m)}$, where
  $u_{j,b}(m) \in \Z_{\geq 0}$, then
\begin{equation}    \label{Xizm}
\langle X_i(z),m \rangle = \prod_{j\in I,b\in\CC^\times} 
\exp \left(u_{j,b}(m)\sum_{r > 0} z^r b^{-r}
\frac{\wt{C}_{i,j}(q^r)}{r}\right).
\end{equation}

Now consider the monomials appearing in the $q$-character of the
simple $U_q(\wh{\mathfrak{g}})$-module $L(m)$. According to \cite{Fre2},
they all have the form
		$$M = m \cdot A_{i_1,a_1}^{-1}A_{i_2,a_2}^{-1}\cdots A_{i_N,a_N}^{-1},$$
where $i_1,\cdots, i_N\in I$ and $a_1,\cdots, a_N\in\CC^\times$. From
the above formulas it follows that
$$\langle X_i(z),M \rangle = \langle X_i(z),m \rangle \cdot
\prod_{1\leq k\leq N, i_k = i} (1 - z a_k^{-1}).$$

\subsection{$W$-invariance of the pairing}

Let us introduce the following series for $i\in I$:
$$U_i(z) = \text{exp}\left( \sum_{m > 0} z^{-m} \frac{h_{i,-m}}{[m]_{q_i}}  \right)$$
so that we have
$$\phi_i^-(z) = k_i^{-1}\frac{U_i(zq_i)}{U_i(zq_i^{-1})}.$$

\begin{rem} A series $\mathcal{P}_i(z)$ of Cartan-Drinfeld elements is defined in \cite[Lemma 12.2.7]{CP} in a different way, but it 
satisfies $\phi_i^-(z) = k_i^{-1}\mathcal{P}_i^-(q_i^{-2}z)/\mathcal{P}_i^-(z)$ and $\mathcal{P}_i(0) = 1$. 
And so $\mathcal{P}_i^-(z) = (U_i(zq_i))^{-1}$.
\end{rem}

Then we have
$$U_i(z) = \prod_{j\in I} \text{exp}\left( \sum_{m > 0} z^{-m} C_{j,i}(q^m)\frac{[d_j]_q [m]_{q_j}}{[d_i]_q [m]_{q_i}}  \frac{\wt{h}_{j,-m}}{[d_j]_q[m]_{q_j}}  \right).$$
But 
$$C_{j,i}(q^m)\frac{[d_j]_q[m]_{q_j}}{[d_i]_q [m]_{q_i}} = C_{i,j}(q^m)\frac{[d_i]_{q^m}}{[d_j]_{q^m}} \frac{[d_j]_q[m]_{q_j}}{[d_i]_q [m]_{q_i}} = C_{i,j}(q^m).$$
Hence
\begin{multline}\label{uiz}U_i(z^{-1}) = \\ X_i(zq_i)X_i(zq_i^{-1})\prod_{j,C_{i,j} = -1}X_j(z)\prod_{j,C_{i,j} = -2}X_j(zq^{-1})X_j(zq)\prod_{j,C_{i,j} = -3}X_j(zq^{-2})X_j(z)X_j(zq^2).\end{multline}

\begin{rem} Note that in comparison to (\ref{ay}), the coefficients $C_{i,j}$ appear, not $C_{j,i}$. However, this is not the expression of the 
root monomials $A_{i,a}$ for the Langlands dual Lie algebra. Rather,
it can be identified with the expression of the variables $\Lambda_{i,a}$ in \cite[Section 9.2]{Hsh}.\end{rem}

\begin{example} In the case $\g = sl_2$, we have $U_1(z) = \on{exp}\left(
    \sum_{m > 0} z^{-m} \frac{h_{1,-m}}{[m]_q}\right)$.\end{example}

Now we define an automorphism $T_i$ of $U_q(\wh{\Hlie})$, and hence of
$U_q(\wh{\Hlie})[[z^{\pm 1}]]$, by setting
$$T_i(\wt{h}_{j,m}) = \wt{h}_{j,m} - \delta_{i,j} q_i^{-m} [d_i]_q h_{i,m}\text{ and }T_i(\wt{k}_j) = \wt{k}_jk_i^{-\delta_{i,j}}$$
for $i,j\in I$ and $m\in\mathbb{Z}\setminus\{0\}$, so that 
\begin{equation}    \label{TY}
  T_i(Y_j(z)) = Y_j(z) (\phi_i^-(z^{-1}q_i^{-1}))^{-\delta_{i,j}}.
\end{equation}
 This implies that the homomorphism $\mathcal{I}$ given by
  formula \eqref{identi} is $W$-equivariant. It also follows that
$$T_i(X_j(z)) = X_j(z) U_i(z^{-1}q_i^{-1})^{-\delta_{i,j}}.$$
The results cited above imply that the operators $T_i, i \in I$, satisfy the
braid relations, and therefore we can define operators $T_w$ for $w\in
W$ by using a reduced decomposition of $w$.


\begin{prop}\label{equith} The pairing $\langle \cdot,\cdot \rangle$ is
  $W$-invariant in the following sense:
  $$\langle T_w(h(z)),T_w(m) \rangle = \langle h(z),m \rangle$$
for all $w\in W$, $h(z)\in U_q(\wh{\Hlie})[[z]]$ and
$m\in\mathcal{M}$.
\end{prop}

\begin{proof}  If $h(z) = k_i^{\pm 1}$, then
    $W$-invariance follows from formula \eqref{pairi} and
    $W$-invariance of $(\cdot,\cdot)$. So it suffices to check that
  for $i,j,k\in I$ and $a\in\mathbb{C}^\times$, we have
$$\langle T_k(\wt{h}_{i,-m}),T_k(Y_{j,a}) \rangle = \langle
    \wt{h}_{i,-m},Y_{j,a} \rangle.$$
The left hand side is equal to the right hand side $\langle
    \wt{h}_{i,-m},Y_{j,a} \rangle$ plus the sum of the term
$$\langle\wt{h}_{i,-m},A_{j,aq_j}^{-\delta_{k,j}}\rangle =
-\delta_{k = j = i} (a q_i)^{-m}[d_i]_q \frac{[m]_{q_i}}{m}$$
and the term
$$-\delta_{i,k}q_i^{{m}}[d_i]_q\langle h_{i,-m},Y_{j,a}A_{j,aq_j}^{-\delta_{k,j}}\rangle
= \delta_{k = j = i} q_i^{m}[d_i]_q (aq_i^2)^{-m}
\frac{[m]_{q_i}}{m}.$$
This gives us the desired result because the sum of these last two
terms in zero.
\end{proof}

For $w\in W$ and $i\in I$, we introduce the following formal Taylor
power series in $z$:
$$X_{w(\omega_i)}(z) := T_w(X_i(z)).$$
 By Lemma \ref{wn}, $T_w(X_i(z))$ depends only on
  $w(\omega_i) \in P$ (and not
separately on $w \in W$ and $i \in I$), so this notation is
justified. We will call these series the {\em $X$-series}.

\begin{example}\label{ttild} (i) For $\Glie = B_2$, we have 
$$X_{s_2(\omega_1)}(z) = X_1(z)
\text{ and }X_{s_2(\omega_2)}(z) = X_2^{-1}(zq^{2})X_1(zq^{2})X_1(z).$$

(ii) More generally, for a simple reflection $s_i$, one has 
$$X_{s_i(\omega_j)}(z) = X_i(z)\text{ if $j\neq i$,}$$
$$X_{s_i(\omega_i)}(z) = X_i(z)U_i(z^{-1}q_i^{-1})^{-1} = (X_i(zq_i^2))^{-1} \left(\prod_{j\in I,C_{i,j} = -1}X_j(zq_i)\right)
\left(\prod_{j\in I, C_{i,j} = -2}X_j(zq^{2})X_j(z)\right)$$
$$\times \left(\prod_{j\in I, C_{i,j} = -3}X_j(zq^{3})X_j(zq)X_j(zq^{-1})\right).$$
\end{example}

\begin{prop}\label{taw} For $i,j\in I$, $a\in\mathbb{C}^\times$ and
  $w\in W$, we have
$$\langle X_{w(\omega_i)}(z), (A_{k,a}^w)^{-1} \rangle = (1 -
  za^{-1})^{\delta_{i,k}}.$$
Moreover, we have 
$$X_{w(\omega_i)}(z) = \prod_{j\in I, a\in \mathbb{C}^\times} X_j(za)^{m_{j,a}^i}$$
where the $m_{j,a}^i\in\mathbb{Z}$ are defined by
$$A_{i,1} = \prod_{j\in I,a\in\mathbb{C}^\times} (A_{j,a^{-1}}^w)^{m_{i,a}^j}.$$
\end{prop}

\begin{proof} By $W$-invariance of the pairing and
    formula \eqref{XA}, we have 
$$(X_{w(\omega_i)}(z), (A_{k,a}^w)^{-1}) = \langle X_i(z),A_{k,a}^{-1}
  \rangle = (1 - z a^{-1})^{\delta_{i,k}}.$$
Next, we write 
$$X_{w(\omega_i)}(z) = \prod_{j\in I, a\in \mathbb{C}^\times} X_j(za)^{m_{j,a}^i}$$
where $m_{j,a}^i\in\mathbb{Z}$. Let us write for $k\in I$
$$A_{k,1}^{-1} = \prod_{l\in I,d\in\mathbb{C}^\times} (A_{l,d^{-1}}^w)^{-w_{l,d}^k},$$
where $w_{l,d}^k\in\mathbb{Z}$. Then
$$ \langle X_{w(\omega_i)}(z), A_{k,1}^{-1} \rangle
=  \prod_{d\in\mathbb{C}^\times} (1 - zd^{-1})^{w_{i,d}^k} 
=  \prod_{a\in \mathbb{C}^\times} (1 - za^{-1})^{m_{k,a}^i}.$$
This implies that $w_{i,d}^{k} = m_{k,d}^i$.
\end{proof}

We also define
$$Y_{w(\omega_i)}(z) := T_w(Y_i(z)) = T_w(\overline{k}_i^{-1})
\frac{X_{w(\omega_i)}(zq_i^{-1})}{X_{w(\omega_i)}(zq_i)} 
=  T_w(\overline{k}_i^{-1})  \prod_{j\in I, a\in\mathbb{C}^\times}
\frac{(X_j(z a q_i^{-1}))^{m_{j,a}^i}}{(X_j(zaq_i))^{m_{j,a}^i}}.$$

Recall the injective group homomorphism $\mathcal{I}$ given by formula
(\ref{identi}). The next lemma follows from the
  $W$-equivariance of $\mathcal{I}$ and formulas \eqref{TY1},
  \eqref{TY}.

\begin{lem}\label{readpow} For any $i\in I$, we have 
\begin{equation}    \label{YY}
  \mathcal{I}(Y_{w(\omega_i),a}) = Y_{w(\omega_i)}(za).
\end{equation}
\end{lem}

Next, we compute the series $X_{w(\omega_i)}(z)$. In
  order to do this, note that $X_{w(\omega_i)}(z)$ is
characterized by the formula
$$Y_{w(\omega_i)}(z) = T_w(\overline{k}_i^{-1}) 
X_{w(\omega_i)}(zq_i^{-1})/X_{w(\omega_i)}(zq_i),$$
which is analogous to the formula
$$
\sigma(Y_{w(\omega_i),a^{-1}}) = [w(\omega_i)]
\Psib_{w(\omega_i),aq_i^{-1}}/\Psib_{w(\omega_i),aq_i},
$$
where $\sigma$ (as defined by formula \ref{sigm}) satisfies $\sigma(Y_{j,b}) =
Y_{j,b^{-1}}$, (see \cite[Proposition 4.5]{FH4}).
The $\ell$-weights $\Psib_{w(\omega_i),a}$ were defined in \cite{FH4}
(this definition will be recalled in Section \ref{defipsiw}).

Comparing these formulas and using Lemma \ref{readpow}, we obtain the
following result. Let ${\mc
  M}'$ be the group of monomials in $\Yim' = \ZZ[\Psib_{i,a}]_{i \in
  I, a \in \C^\times}$.

\begin{prop}\label{tpsi}
  (1) The homomorphism $\mathcal{I}$ extends uniquely to a homomorphism
\begin{equation}\label{identip}\mathcal{I}' : \mathcal{M}' \rightarrow
  (U_q(\wh{\Hlie})[[z]])^\times\end{equation}
such that
\begin{equation}\label{extia}
  {\mc I}'(\Psib_{i,a}) = X_i(za^{-1}) \qquad i\in I,
  a\in\mathbb{C}^\times.
\end{equation}

(2) For $i\in I$, $a\in \mathbb{C}^\times$ and $w\in W$, we have 
$$\mathcal{I}'(\Psib_{w(\omega_i),a}) = X_{w(\omega_i)}(za^{-1}).$$
\end{prop}

This means that the powers of $X_j(zb)$ in the factorization of 
$X_{w(\omega_i)}(z)$ are the same as the powers of the
$\Psib_{j,b^{-1}}$ in the factorization of
$\Psib_{w(\omega_i),1}$. The latter are given in Section
\ref{defipsiw}. This implies the following.

\begin{prop}\label{factW}
The series $X_{w(\omega_i)}(z)$ are obtained from the
factorization of $Y_{w(\omega_i),1}$ in the variables $Y_{k,a}^{\pm 1}$ by replacing 

$Y_{k,a}$ by $X_k(za)$ if $d_i = d_k$, 

$Y_{k,a}$ by $X_k(zaq)X_k(zaq^{-1})$ if $C_{i,k} = -2$,

$Y_{k,a}$ by $X_k(zaq^{-2})X_k(za)X_k(zaq^2)$ if $C_{i,k} = -3$,

$W_{k,a}$ by $X_k(za)$ if $C_{k,i} < - 1$, where $W_{i,a}$ is given by
(\ref{wia}).
\end{prop}

\section{Extremal monomial property and polynomiality}\label{tmproof}

In this section, we relate the extremal monomial property of
$q$-characters, Conjecture \ref{mainc} (resp. a weaker property,
Theorem \ref{pma}) to the polynomiality (resp. rationality) of the
eigenvalues of the renormalized $X$-series $X^N_{w(\omega_i)}(z)$
acting on simple finite-dimensional $U_q(\ghat)$-modules. We first
derive certain identities between the eigenvalues of the renormalized
$X$-series on a simple finite-dimensional $U_q(\ghat)$-module
(Corollary \ref{inveigw}). Then we show that our Theorem \ref{pma}
implies rationality of the eigenvalues of $X^N_{w(\omega_i)}(z)$ on
these modules (Theorem \ref{thimp}), from which we derive that the
operator $X^N_{w(\omega_i)}(z)$ itself is an expansion of a rational
function in $z$ (Theorem \ref{rattw}). We conjecture that it is in
fact a polynomial in $z$ (Conjecture \ref{twpol}).  In \cite{FH} we
proved this polynomiality for $w=e$, and in \cite{Z} it was proved for
$w=w_0$. Here, we prove polynomiality of the eigenvalues of
$X^N_{w(\omega_i)}(z)$ when $w$ is a simple reflection (Theorem
\ref{symsimp}). This implies Conjecture \ref{mainc} for simple
reflections.

In the next section (Section \ref{trans}), we will introduce the
generalized Baxter operators $t_{w(\omega_i)}(z,u)$ (they are the dual
versions of the Baxter operators we had previously introduced in
\cite{FH4}) and show that the $X$-series $X_{w(\omega_i)}(z)$ are
certain limits of these operators. Therefore we obtain a link between
polynomiality of the eigenvalues of $X^N_{w(\omega_i)}(z)$ and
polynomiality of the eigenvalues of these Baxter operators (properly
renormalized). This has important consequences for the quantum
integrable models of XXZ type associated to $U_q(\ghat)$; namely, it
allows one to give different descriptions (one for each $w \in W$) of
the spectra of the commuting Hamiltonians in these models in terms of
the polynomial eigenvalues of the renormalized Baxter operators
$t_{w(\omega_i)}(z,u)$. For $w=e$ this description has been proved in
\cite{FH} (see Section \ref{nver}), and for general $w \in W$ this is
our Conjecture \ref{trm}.

\subsection{Rationality of $X^N_{w(\omega_i)}(z)$}

Let $L(m)$ be the simple $U_q(\ghat)$-module with the highest monomial
$m$. Recall that $\langle X_{w(\omega_i)}(z),T_w(m)
  \rangle$ is the eigenvalue of $X_{w(\omega_i)}(z)$ on the $\ell$-weight
corresponding to the monomial $T_w(m)$.
From the $W$-invariance of the pairing $\langle
  \cdot,\cdot \rangle$ (see Proposition \ref{equith}), we obtain the
  following key result.

\begin{thm} For any $i\in I$, the pairing
$\langle X_{w(\omega_i)}(z),T_w(m) \rangle$ is independent of $w\in W$.
\end{thm}

This implies the following statement.

\begin{cor}\label{inveigw}
The eigenvalue of $X_{w(\omega_i)}(z)$ on the $\ell$-weight
corresponding to the monomial $T_w(m)$ is independent of $w\in W$.
\end{cor}

Let us denote this eigenvalue by $f_{i,m}(z)$. Thus, we have
$$
\langle X_{w(\omega_i)}(z),T_w(m) \rangle = f_{i,m}(z), \qquad \forall
w \in W.
$$

In \cite{FH}, we computed it in the case $w=e$, and the result is
  the RHS of formula \eqref{Xizm}:
\begin{equation}    \label{fimz}
f_{i,m}(z) = \prod_{j\in I, b\in\mathbb{C}^\times} \text{exp}\left(
u_{j,b}(m)\sum_{r > 0} (zb)^r \frac{\wt{C}_{i,j}(q^r)}{r}
\right),
\end{equation}
where the numbers $u_{j,b}(m) \in \Z_{\geq 0}$ are the
  powers appearing in the formula
  $$
  m = \prod_{j\in I, b\in\CC^\times} Y_{j,b}^{u_{j,b}(m)}.
  $$

Now let us define the  {\em renormalized $X$-series} acting on
  the simple module $L(m)$ by the formula
\begin{equation}    \label{XN}
  X_{w(\omega_i)}^N(z) := (f_{i,m}(z))^{-1} X_{w(\omega_i)}(z).
\end{equation}

 Thus, $X_{w(\omega_i)}^N(z)$ is a formal Taylor power
  series in $z$ of the form
   $$
   X_{w(\omega_i)}^N(z) = \sum_{k\geq 0} X_{w(\omega_i),k}^N \; z^k
   $$
   whose coefficients $X_{w(\omega_i),m}^N$ are elements of the
   commutative algebra $U_q(\wh{\Hlie})$. One can derive from this
   that any finite-dimensional
   $U_q(\ghat)$-module $L(m)$ has a basis (over $\C$) in which
     the matrix of
   the formal power series $X_{w(\omega_i)}^N(z)$ is upper-triangular,
   with the diagonal entries being the eigenvalues of the operator
   $X_{w(\omega_i),m}^N$ acting on $L(m)$.

   \begin{example}
     Here are two examples of such matrices:
     $$\begin{pmatrix}z+z^2 & z^3 \\ 0 & z+z^2 \end{pmatrix} \qquad
     , \qquad
     \begin{pmatrix}z+z^2 & \ds \sum_{k\geq 0} \frac{z^k}{k!}
       \\ 0 & z+z^2 \end{pmatrix}$$
   In both cases, the single eigenvalue is a polynomial in $z$ (namely,
   $z+z^2$). While in the
   first case the matrix itself is also a polynomial in $z$, in
   the second case it is not (not even an expansion of a rational
   function in $z$;  in fact, on the intersection of the
     first row and second column we have the expansion of
     $\on{exp}(z)$, which is {\em not} an expansion of a rational function).

     Here are two more examples:
     $$
     \begin{pmatrix}z+z^2 & \ds \sum_{k\geq 0} z^k & 0 & 0 \\
       \\ 0 & z+z^2 & 0 & 0 \\ 0 & 0 & \sum_{k\geq 0} k z^k & z^5 \\
       0 & 0 & 0 & \sum_{k\geq 0} k z^k\end{pmatrix} \quad , \quad
     \begin{pmatrix}z+z^2 & \ds \sum_{k\geq 0} z^k & 0 & 0 \\
       \\ 0 & z+z^2 & 0 & 0 \\ 0 & 0 & \sum_{k\geq 0} k z^k &
       \sum_{k\geq 0} k! z^k  \\ 
     0 & 0 & 0 & \sum_{k\geq 0} k z^k\end{pmatrix}
     $$
In both cases, the matrix has two eigenvalues, which are
expansions of two rational
functions (namely, $z+z^2$ and $z/(1-z)$). But while in the first case
{\em all} entries of the matrix are expansions of rational functions, in the
second case this is not so (indeed, the entry on the intersection of the
3rd row and the 4th column is {\em not} an expansion of a rational
function).
\end{example}

If all entries of the matrix of a given operator are polynomials
(resp. expansions of rational functions) in $z$, we will say that this
{\em operator is a polynomial (resp. an expansion of a rational
  function) in $z$}. Note that this property does not depend on the
basis.

As the next theorem shows, the eigenvalues of the renormalized operator
  $X_{w(\omega_i)}^N(z)$ on any simple finite-dimensional
  $U_q(\ghat)$-module are expansions of rational functions in
    $z$.
   
  \begin{thm}    \label{thimp}
  Every eigenvalue of the renormalized operator
  $X_{w(\omega_i)}^N(z)$ acting on any simple finite-dimensional
  $U_q(\ghat)$-module is an expansion of a rational function in
    $z$.
  \end{thm}

  \begin{proof} This follows immediately from Proposition \ref{taw}
    and Theorem \ref{pma}.
    \end{proof}

  \begin{rem} It follows from the results of \cite{FH}
    that the degree of the rational function expressing the eigenvalue of
  $X_{w(\omega_i)}^N(z)$ on the $\ell$-weight space corresponding to a 
	given monomial $m'$ appearing in the $q$-character $\chi_q(L(m))$ equals
  the multiplicity of $w(\alpha_i)$ in the
  decomposition of the ordinary weight of the monomial
  $T_w(m)(m')^{-1}\in\mathcal{M}$.

  More precisely, this is established
  in \cite{FH} for $w = e$. For a general $w\in W$, by Proposition
  \ref{equith}, we have
	$\langle X_{w(\omega_i)}(z),m'\rangle = \langle X_i(z) ,
  T_w^{-1}(m')\rangle$ and so
\begin{equation}\label{jdegree}\langle X_{w(\omega_i)}^N(z),m'\rangle
  = \langle X_i^N(z) , T_w^{-1}(m')\rangle = \langle X_i(z),
  m^{-1}T_w^{-1}(m') \rangle.\end{equation}
Here $m^{-1}T_w^{-1}(m')\in \mathbb{Z}[A_{j,b}^{\pm 1}]_{j\in I,
  b\in\mathbb{C}^*}$ and therefore it follows from (\ref{XY}) that
the degree of the rational function $\langle X_i(z),
m^{-1}T_w^{-1}(m') \rangle$ is the multiplicity of $\alpha_i$ in the
ordinary weight of $m T_w^{-1}((m')^{-1})$, which is equal to the
weight written above.
\end{rem}

In the next theorem, we claim the rationality of the operator
$X_{w(\omega_i)}^N(z)$ itself. Hence it is a stronger
statement than Theorem \ref{pma}.

\begin{thm}\label{rattw} The renormalized $X$-series
  $X_{w(\omega_i)}^N(z)$ acting on any simple finite-dimensional
  $U_q(\ghat)$-module is an expansion of a rational function in $z$.
\end{thm}

\begin{proof} By \cite[Theorem 5.17]{FH}, $X_{\omega_i}^N(z)$ is a
  polynomial in $z$.\footnote{ Note that the sentence ``By
    Proposition 5.5, Corollary 5.10 implies:'' before Theorem 5.17 in
    \cite{FH} should read ``By Proposition 5.5, Theorem 5.9
    implies:''.}  Hence the statement of the theorem is true for $w =
  e$. Now, for a general $w \in W$, recall that we have established in
  Proposition \ref{taw} that $X_{w(\omega_i)}(z)$ is a Laurent
  monomial in the $X_j(zb) = X_{\omega_j}(zb), j \in I, b \in
  \C^\times$. By the above result, the matrix of
    $X_j(zb)$ on $L(m)$ is equal to its eigenvalue on the weight
    subspace of $L(m)$ associated to the weight $w(\varpi(m))$ times a
    rational matrix-valued function in $z$. This implies that
    $X_{w(\omega_i)}(z)$ is equal to its eigenvalue on the
    (one-dimensional) weight
    subspace of $L(m)$ associated to the weight $w(\varpi(m))$ times a
    rational matrix-valued function in $z$. But as we have shown
    above, the latter eigenvalue is equal to $f_{i,m}(z)$. Therefore
    the matrix of $X_{w(\omega_i)}^N(z)$, which is given by formula
    \eqref{XN}, acting on $L(m)$, is an expansion of a rational
    function in $z$.
\end{proof}

\begin{rem} We note that rationality of closely related power series,
  which can be expressed as certain products of shifts
  of the $X$-series $X_i(z)$ acting on finite-dimensional
  $U_q(\ghat)$-modules (not necessarily those of highest weight)
  was recently established in \cite{gtl}.
\end{rem}

\subsection{Polynomiality of $X^N_{w(\omega_i)}(z)$}    \label{polyn}

Here is the second main conjecture of this paper.
  
\begin{conj}\label{twpol}  The renormalized operator
  $X_{w(\omega_i)}^N(z)$ acting on any simple finite-dimen\-sional
  $U_q(\ghat)$-module $L(m)$ is a polynomial in $z$.
\end{conj}

The following conjecture is a weaker version of
  Conjecture \ref{twpol}, as it states polynomiality of the
  {\em eigenvalues} of $X_{w(\omega_i)}^N(z)$.

\begin{conj}    \label{twpol1}
Every eigenvalue of the renormalized $X$-series
  $X_{w(\omega_i)}^N(z)$ acting on any simple finite-dimensional
  $U_q(\ghat)$-module $L(m)$ is a polynomial in $z$.
\end{conj}

\begin{rem} Note that although this statement of Conjecture \ref{twpol1}
  in the case $w = e$ was proved in \cite{FH},
  equation (\ref{jdegree}) by itself does not imply that
$\langle X_{w(\omega_i)}^N(z),m'\rangle$ is a polynomial because for a
  general monomial $m'$ occurring in $\chi_q(L(m))$, the monomial
  $T_w^{-1}(m')$ does not necessarily occur in $\chi_q(L(m))$.\end{rem}

Now we relate these conjectures to Conjecture \ref{mainc}.

\begin{thm}    \label{equi}
Conjecture \ref{mainc} is equivalent to Conjecture \ref{twpol1}, which
in turn follows from Conjecture \ref{twpol}.

More precisely, the statement of Conjecture \ref{mainc} for the simple
module $L(m)$ and $w \in W$ is equivalent to the statement of
Conjecture \ref{twpol1} for $X_{w(\omega_i)}^N(z), i \in I$, and
$L(m)$.
\end{thm}

\begin{proof} As in the proof of Theorem \ref{thimp}, Proposition
  \ref{taw} proves the equivalence of Conjectures \ref{mainc} and
  \ref{twpol1}.

  Conjecture \ref{mainc} claims
  polynomiality of the eigenvalues of $X_{w(\omega_i)}^N(z)$, whereas
  Conjecture \ref{twpol} claims polynomiality of the operator
  $X_{w(\omega_i)}^N(z)$ itself. Hence Conjecture \ref{mainc} follows
  from Conjecture \ref{twpol}.
\end{proof}

For $w = e$, Conjectures \ref{mainc} and \ref{twpol} were established
in \cite{FH}. For $w = w_0$, the longest element of the Weyl group,
Conjecture \ref{twpol} was established in \cite{Z}. By Theorem
\ref{equi}, this also proves Conjecture \ref{twpol}. Thus, we obtain
the second proof of Conjecture \ref{mainc} for $w=w_0$ (the first one
is given in the proof in Theorem \ref{long}). In the next section, we
will prove Conjecture \ref{mainc} for simple reflections.

\subsection{Polynomiality for simple reflections}

 Let now $w = s_i, i \in I$, a simple reflection from the
  Weyl group.

\begin{thm}\label{symsimp} Conjecture \ref{twpol1} holds for simple
  reflections $s_i \in W, i \in I$. Namely, for any $j\in I$, every
  eigenvalue of $X_{s_i(\omega_j)}^N(z)$ on a simple
  finite-dimensional $U_q(\wh{\mathfrak{g}})$-module is a
  polynomial in $z$.
\end{thm}

Applying Theorem \ref{equi}, we obtain the following corollary.

\begin{thm}    \label{refl}
  Conjecture \ref{mainc} holds for simple reflections $s_i \in W, i
  \in I$.
\end{thm}

To prove Theorem \ref{symsimp}, we need some results from \cite[Sections 3.3, 3.4]{Hmz}. Let $k\in I$ and $M$ a $k$-dominant monomial
(that is all powers of the variables $Y_{k,a}$ in $M$ are non-negative). An element $L_k(M)$ is defined in \cite{Hmz} in the
following way. Let $M^{(k)}$ be the expression obtained from $M$ by keeping only the factors $Y_{k,a}$. Then it corresponds to a simple module $L$ over $U_{q_k}(\wh{sl}_2)$. 
Its $q$-character belongs to
$M^{(k)} \cdot \mathbb{Z}[(Y_{k,bq_k^{-1}}Y_{k,bq_k})^{-1}]_{b\in\mathbb{C}^\times}$. By
definition,
$$L_k(M)\in M \cdot \mathbb{Z}[(Y_{k,bq_k^{-1}}Y_{k,bq_k})^{-1}]_{b\in\mathbb{C}^\times}$$ 
is obtained from the $q$-character of $L$ by replacing $M^{(k)}$ by $M$ and each factor $(Y_{k,bq_k^{-1}}Y_{k,bq_k})^{-1}$ by $A_{k,b}^{-1}$.

Now let $m$ be a dominant monomial. Then it follows from
\cite[Proposition 3.1]{Hmz} that we have a unique decomposition
\begin{equation}\label{decompk}\chi_q(L(m)) = \sum_{M\text{
      $k$-dominant}} \lambda_k(M) L_k(M)\end{equation}
with $\lambda_k(M)\geq 0$.

We are ready to prove Theorem \ref{symsimp}.

\begin{proof} For $j\neq i$, we have $X^N_{s_i(\omega_j)} = X_j^N(z)$ and
  the result is proved in \cite{FH} (see the last
    paragraph of Section \ref{polyn}). It remains to consider the case
    $j=i$. Then $T_{s_i(\omega_i)}(z)$ is given by explicit formulas
  in Example \ref{ttild}.  Moreover, formula (\ref{mult})
    implies that the desired polynomiality property is multiplicative
    with respect to taking tensor products of finite-dimensional
    $U_q(\ghat)$-modules. Hence it suffices to prove the
  polynomiality of the eigenvalues  of
    $T_{s_i(\omega_i)}(z)$ on the fundamental representations
  $L(Y_{l,1})$, $l\in I$.

First assume that $l\neq i$. Then we have $T_{s_i}(Y_{l,1}) = Y_{l,1}$. Recall that a monomial $m$ in $\chi_q(L(Y_{l,1}))$ is of the form 
\begin{equation}    \label{via}
  m = Y_{l,1}\prod_{k\in I, a\in\mathbb{C}^\times}A_{k,a}^{-v_{k,a}},
\end{equation}
with the $v_{k,a}\geq 0$. By Proposition \ref{taw}, in
  order to prove our statement about the eigenvalues, it suffices to
 show that the powers with which the variables
  $A_{k,a}^{s_i}$ appear in the factorization of the monomial
  $m (Y_{l,1}^{s_i})^{-1}$ are negative.
The formulas in the proof of Lemma \ref{qroot} give a relation between
the $A_{k,a}$ and the $A_{k,a}^{s_i}$. In particular, if $k\neq i$,
then $A_{k,a}^{s_i}$
appears once in the factorization of $A_{k,a}$. This implies that the
$A_{k,a}^{s_i}$ with $k \neq i$ appear in the factorization of the monomial
  $m (Y_{l,1}^{s_i})^{-1}$ in negative powers. Hence, it remains to
consider the case $k = i$ and to prove that the powers
  are negative in this case as well. The explicit
    formulas given in the proof of Lemma \ref{qroot} show that we need to
prove the following inequality: 
\begin{multline}\label{inegv}v_{i,a}\leq \\ \sum_{k\in I,C_{i,k} = -1}v_{k,aq_i^{-1}} + 
\sum_{k\in I, C_{i,k} = -2}(v_{k,aq^{-2}}+v_{k,a}) +
\sum_{k\in I, C_{i,k} =
  -3}(v_{k,aq^{-3}}+v_{k,aq^{-1}}+v_{k,aq}).\end{multline}

Let is prove this inequality by induction on the length of the weight
of $m Y_{l,1}^{-1}$. Formula (\ref{inegv}) is clear if $m$ is the highest
monomial, $m = Y_{l,1}$. Otherwise, it follows from the results of
\cite{Fre2} that
there exists $p \in I$  such that $m$ is not $p$-dominant. Consider
the decomposition (\ref{decompk}):
$$\chi_q(L(Y_{l,1})) =  \sum_{M_p\text{ $p$-dominant}} \lambda_{M_p} L_p(M_p).$$ 
Then there is a $p$-dominant monomial $M$, such that $m$ is a monomial $L_p(M)$ which occurs in this decomposition. 
By the induction hypothesis, we have the desired
  inequality for $M$.

If $p\neq i$, the inequality for $M$ implies directly the inequality for $m$.

If $p = i$, we use explicit formulas for the the $q$-characters of
simple modules over $U_q(\wh{sl}_2)$. These formulas imply that
$$v_{i,a}(mM^{-1})\leq u_{i,aq_i^{-1}}$$
where $u_{i,aq_i^{-1}}$ is the multiplicity of $Y_{i,aq_i^{-1}}$ in
$M$ and $v_{i,a}$ is defined in formula \eqref{via}. If we define $V_{r,b}$ now by the formula 
$$M = Y_{l,1}\prod_{r,b}A_{r,b}^{-V_{r,b}},$$
then we have the following inequalities:
\begin{equation}\label{ineq}u_{i,aq_i^{-1}}\leq -V_{i,a}-V_{i,aq_i^{-2}} + \sum_{r\in I,C_{i,k} = -1}V_{r,aq_i^{-1}}\end{equation}
$$+ 
\sum_{r\in I, C_{i,r} = -2}(V_{r,aq^{-2}}+V_{r,a}) +
\sum_{r\in I, C_{i,r} = -3}(V_{r,aq^{-3}}+V_{r,aq^{-1}}+V_{r,aq}).$$
If $r\neq i$, then $V_{r,b} = v_{r,b}$. Since
$$v_{i,a} = V_{i,a} + v_{i,a}(mM^{-1}) \leq V_{i,a} + V_{i,aq_i^{-2}} + v_{i,a}(mM^{-1}),$$
we obtain the desired inequality for $m$. This completes
  the case $l \neq i$.

Finally, consider the case $l = i$. Then
  $T_{s_i}(Y_{i,1}) = Y_{i,1}A_{i,aq_i}^{-1}$. For any monomial $m$ in $\chi_q(L(Y_{i,1}))$, one of the following three conditions hold:

$m = Y_{i,1}$;

$m = T_{s_i}(Y_{i,1})$;

$m = T_{s_i}(Y_{i,1})\prod_{k\in I, r >  r_i}A_{k,q^r}^{-v_{k,q^r}}$, with $v_{k,q^r}\geq 0$.

The result is clear for the first two monomials. For the third one, we
use the same argument as above. Namely, we need to show that the powers with
which the variables
  $A_{k,a}^{s_i}$ appear in the factorization of the monomial
  $m (Y_{i,1}^{s_i})^{-1}$ are negative. As above, this is clear if $k \neq i$. For $k = i$, we establish 
	the inequality (\ref{inegv}) by induction on the length of the weight
of $m (T_{s_i}(Y_{i,1}))^{-1}$. We have already seen the result for $m = T_{s_i}(Y_{i,1})$.
Otherwise, there is $p\in I$ such that $m$ is not $p$-dominant and there is
a $p$-dominant monomial $M$, such that $m$ is a monomial $L_p(M)$ which occurs in this decomposition of $\chi_q(L(Y_{i,1}))$. 
As $m\neq Y_{i,1}, T_{s_i}(Y_{i,1})$, we have that $M\neq Y_{i,1}$ (as $L_i(Y_{i,1}) = Y_{i,1} + T_{s_i}(Y_{i,1})$ and $L_p(Y_{i,1}) = Y_{i,1}$ for $p\neq i$). 
By the induction hypothesis, we have the desired inequality for $M$. 

If $p\neq i$, the inequality for $M$ implies directly the inequality for $m$.

If $p = i$, we have seen $v_{i,a}(mM^{-1})\leq u_{i,aq_i^{-1}}$ with the same notations as above. If we set $V_{r,b}$ as above, we have also the inequalities (\ref{ineq}).
If $r\neq i$, then $V_{r,b} = v_{r,b}$. Since
$$v_{i,a} = -\delta_{a,q_i} + V_{i,a} + v_{i,a}(mM^{-1}) \leq V_{i,a} + V_{i,aq_i^{-2}} + v_{i,a}(mM^{-1}),$$
we obtain the desired inequality for $m$.
\end{proof}

\section{Transfer-matrices and Baxter polynomiality} \label{trans}

In this section we introduce the transfer-matrices and particular
examples of the transfer-matrices, the so-called generalized Baxter
$Q$-operators. We will show that the $\OO^*$ versions of
  these operators are in fact multi-parameter deformations of the
$X$-series $X_{w(\omega_i)}(z)$. We expect that polynomiality of the
renormalized series $X_{w(\omega_i)}^N(z)$, which we conjectured in
the previous section, extends to polynomiality of the (suitably
renormalized) generalized Baxter $Q$-operators. This is Conjecture
\ref{trm}. So far, it has been proved only for $w=e$, in \cite{FH}. In the
special case $\g=sl_2$, we recover the results of Baxter on the
polynomiality of the eigenvalues of the celebrated Baxter
$Q$-operator. More precisely, Baxter introduced his $Q$-operator in
the XYZ model \cite{Baxter}, whereas we are considering the analogous
operator in the XXZ model (which appears in a certain limit of the XYZ
model) as well as its generalizations to the XXZ-type models
associated to $U_q(\ghat)$. See the Introduction of \cite{FH} for more
details.

\subsection{Definition of transfer-matrices}

The universal $R$-matrix $\mathcal{R}$ of $U_q(\wh{\Glie})$ belongs to
the completed
tensor product $U_q(\wh{\Glie})\wh{\otimes} U_q(\wh{\Glie})$
(it is
completed with respect to the
$\ZZ$-grading of $U_q(\wh{\Glie})$, and with a certain completion of
the tensor product of the Cartan subalgebra).
This completed tensor product acts on the tensor products of modules
we consider below.
The universal $R$-matrix satisfies the Yang-Baxter equation.

In fact, it is known that $\mathcal{R}$ belongs to
$U_q(\wh{\mathfrak{b}})\wh{\otimes} U_q(\wh{\mathfrak{b}}_-)$, where
$U_q(\wh{\mathfrak{b}}_-)$ is the subalgebra generated by $f_i$ 
and $k_i^{\pm1}$ with $0\le i\le n$. Hence for any any
$U_q(\wh{\mathfrak{b}})$-module $V$, we define the $L$-{\em
  operator} associated to $V$ as
$$L_V(z) := (\pi_{V(z)}\otimes \on{Id}) (\mathcal{R}) \in
\on{End}(V)[[z]]\wh{\otimes} U_q(\wh{\mathfrak{g}}),$$ where $\pi_V(z):
U_q(\wh{\mathfrak{b}})\rightarrow \on{End}(V)[[z]]$ is
the action of $U_q(\wh{\mathfrak{b}})$ on $V$ twisted by $\tau_z$.

Let $V$ be a $U_q(\wh{\mathfrak{b}})$-module which is Cartan-diagonalizable 
with finite-dimensional weight spaces (for example, this property holds if $V$
is in the category $\mathcal{O}$ or
in the dual category $\mathcal{O}^*$). Assume in addition that
the ordinary weights of $V$ are in $P \subset P_\Q$. For $g\in
U_q(\wh{\mathfrak{b}})$ (or in
$\on{End}(V)$), the twisted trace of $g$ on $V$ is defined as follows:
$$\on{Tr}_{V,u}(g) = \sum_{\lambda \in P}
\on{Tr}_{V_\lambda}(\pi_V(g))\left(\prod_{i\in I}u_i^{\lambda_i}\right) \in
\CC[[u_i^{\pm 1}]]_{i\in I},$$ where $u = (u_i), i \in I$, where each
$u_i$ is a formal variable and $\lambda_i\in\ZZ$ is defined by
$\lambda(i) = q_i^{\lambda_i}$.

\begin{defi} \label{gr tr}
The twisted transfer-matrix associated to $V$ is
$$t_V(z,u) := (\on{Tr}_{V,u}\otimes \on{Id}) (L_V(z))\in U_q(\wh{\mathfrak{b}}_-)[[z, u_i^{\pm 1}]]_{i\in I}.$$
\end{defi}

For $i\in I$, introduce the variable
$$v_i := \prod_{j\in I} u_j^{C_{j,i}}$$ 
which corresponds to the $i$th simple root. Then we have
$$t_V(z,u) \in 
  U_q(\wh{\mathfrak{b}}_-)[[z,v_i^{-1}]]_{i\in I}[u_j^{\pm 1}]_{j\in
    I} \text{ if $V$ is in $\mathcal{O}$,}$$
$$t_V(z,u) \in  U_q(\wh{\mathfrak{b}}_-)[[z,v_i]]_{i\in
  I}[u_j^{\pm 1}]_{j\in I} \text{ if $V$ is in $\mathcal{O}^*$,}$$
Hence products of the twisted transfer-matrices of modules
  from the same category ($\OO$ or $\OO^*$) are well-defined.



If $V$ is a finite-dimensional module (i.e. an object of
  $\mathcal{C}$), $t_V(z,u)$ is a polynomial in the variables
  $u_i^{\pm 1}$, and therefore the variables $u_i$
  can be specialized to any non-zero complex values. For example, the
  standard specialization is
$$u_i = q_i^{2\sum_{j\in I}(C^{-1})_{j,i}} = q^{2\sum_{j\in I}(DC^{-1})_{i,j}}.$$ 
This means $v_i = q_i^2$ because $DC^{-1}$ is symmetric (and $u_1 = q$
if $\g = sl_2$).

It can be proved as in \cite[Lemma 2]{Fre} that for $V,V'$ in the
category $\mathcal{O}$ (or $\mathcal{O}^*$) whose weights are in
$P$, and for any extension $W$  of $V$ and $V'$, we have
\begin{equation}\label{alg}t_{W}(z,u) = t_V(z,u) + t_{V'}(z,u)\text{
    and }t_{V\otimes V'}(z,u) = t_V(z,u) t_{V'}(z,u).\end{equation}
Hence $t_V(z,u)$ depends only of the class of
  $V$ in the corresponding Grothendieck ring.

\begin{thm}\cite{FH} For $V, V'$ in the category $\mathcal{O}$ (or
  $\mathcal{O}^*$)
  whose weights are in $\overline{P}\subset \tb^\times$, we have
$$t_V(z,u)t_{V'}(w,u) = t_{V'}(w,u) t_V(z,u).$$\end{thm}

\subsection{$X$-series as limits of the transfer-matrices}    \label{Rgeq}

Recall the notation $Q^+$ from section \ref{Liealg}.
Let $V$ be an
  object of $\OO$ (resp. $\OO^*$) whose ordinary weights are in
  $Q^+$ (resp. $-Q^+$). Let us write
$$t_V(z,u) = \sum_{m\geq 0} t_V[m](u) z^m.$$ Then the Fourier
  coefficients $t_V[m](u)$, where $m\geq 0$ and $V$ runs over the
  classes of objects of the category $\mathcal{O}$ (resp.
  $\mathcal{O}^*$) satisfying the above condition, are mutually
  commuting elements of the algebra $U(\wh\bb_-)[[v_i^{-1}]]_{i\in
    I}$ (resp. $U(\wh\bb_-)[[v_i]]_{i\in I}$). Hence they generate a
  commutative subalgebra of this algebra. It follows that every
  element of this subalgebra, as well as the algebra itself, have a
  well-defined limit when $v_i^{-1} = 0$ $(i \in I)$ (resp. $v_i
  = 0$ $(i \in I)$).

  In particular, let $V=R_{i,a}^+$, which is an object of $\OO^*$
  satisfying the above condition. We are interested in the limit of
  $t_{R_{i,a}^+}(z,u)$ when $v_j = 0, \forall j \in I$, which we will
  express simply as $v=0$. This limit was computed in \cite[Theorem
    5.5]{FH}.

\begin{thm}    \label{limit}
For each $i \in I$, the limit of $t_{R_{i,a}^+}(z,u)$ at $v = 0$ is
equal to $X_i(za)$. 
\end{thm}

Thus, we see the transfer-matrices $t_{R_{i,a}^+}(z,u)$ simplify
dramatically in the limit $v = 0$: all of their Fourier coefficients
end up in the (commutative) Cartan-Drinfeld subalgebra
$U_q(\wh{\mathfrak h}_-)$ of $U_q(\wh\bb_-)$!

The results of \cite{FH} allow us to describe the eigenvalues of both
$t_{R_{i,a}^+}(z,u)$ and $X_i(za)$ on any $U_q(\ghat)$-module
$N$ which is a tensor product of simple finite-dimensional modules. It
is instructive to compare these eigenvalues.

The module $N$ has a highest weight vector $x$ whose
$\ell$-weight corresponds to a monomial $m = \prod_{j\in I,
  b\in\CC^\times}Y_{j,b}^{u_{j,b} (m)} \in \Yim$ with the ordinary
weight $\omega$. Therefore $x$ is an eigenvector of 
  $X_i(z)$. The corresponding eigenvalue $f_{i,m}(z)$ was found in \cite{FH}; it is given by formula
\eqref{fimz}. Moreover, by \cite[Theorem 4.1]{Fre2}, all other $\ell$-weights
of $N$ correspond to monomials of the form
\begin{equation}    \label{MA}
  M = m A_{i_1,a_1}^{-1}A_{i_2,a_2}^{-1}\cdots A_{i_N,a_N}^{-1}
\end{equation}
where $i_1,\cdots, i_N\in I$ and $a_1,\cdots, a_N\in\CC^\times$. 

This implies the following description the eigenvalues of $X_i(z)$

\begin{prop}\label{firstpol} The eigenvalue of $f_{i,m}(z)^{-1} X_i(z)$ on the
  $\ell$-weight subspace of $N$ corresponding to the $\ell$-weight $M$
  is equal to
\begin{equation}    \label{PX}
  \prod_{1\leq k\leq N, i_k = i} (1 - z a_k^{-1}).
\end{equation}
\end{prop}

Now let us discuss the eigenvalues of $t_{R_{i,a}^+}(z,u), i\in I, a
\in \C^\times$, on $N$. It is easy to see that $t_{R_{i,a}^+}(z,u) =
t_{R_{i,1}^+}(za,u)$, so it is enough to consider the eigenvalues of
$t_{R_{i,1}^+}(z,u), i\in I$.


Let $\lambda$ be a weight of
$N$ and $ht_i(\omega -
\lambda)$ the multiplicity of $\alpha_i$ in $\omega - \lambda$.


\begin{thm}\cite{FH}\label{catOpol} The operator
$$f_{i,m}(z)^{-1}t_{R_{i,1}^+}(z,u)\in
((\on{End}(N_\lambda))[[v_j]]_{j\in I})[z]$$
is a polynomial in $z$ of degree $ht_i(\omega - \lambda)$. 
\end{thm}

In the limit $v=0$, the operator $f_{i,m}(z)^{-1}t_{R_{i,1}^+}(z,u)$
restricted to the subspace $N_\lambda \subset N$ becomes the operator
$(f_{i,m}(z))^{-1}X_i(z)$ restricted to the subspace $N_\lambda
\subset N$. According to Theorem \ref{catOpol}, every eigenvalue of
the former operator is a polynomial in $z$ of degree $ht_i(\omega -
\lambda)$. Moreover, it is known that its roots (generically)
correspond to solutions of the corresponding Bethe Ansatz equations
(see \cite{FH}). On the other hand, the eigenvalues of the latter
operator are given by the polynomial \eqref{PX} of the same degree
$ht_i(\omega - \lambda)$. The roots of this polynomial can be found
directly from the corresponding $\ell$-weight (see formula
\eqref{MA}).

Thus, we obtain that in the limit $v \to 0$ each eigenvector of
$f_{i,m}(z)^{-1}t_{R_{i,1}^+}(z,u)$ tends to an $\ell$-weight
vector, while the corresponding set of solutions of the Bethe Ansatz
equations (describing the eigenvalue) tends to the numbers $a_i^{-1}$
in formula \eqref{PX} encoded by this $\ell$-weight according to
formula \eqref{MA}.


We are going to introduce a large family of modules from the
category $\OO^*$ labeled by $w \in W$ and $i \in I$ (with the modules
corresponding to $w=e$ being $R_{i,a}^+, i \in I$) for
which we will conjecture the same polynomiality property (Conjecture
\ref{trm}).


\subsection{The $\ell$-weights $\Psib_{w(\omega_i),a}$}\label{defipsiw}

The first step is to recall the definition of the $\ell$-weights
$\Psib_{w(\omega_i),a}, i \in I, w \in W$, from \cite{FH4}. The transfer-matrices
corresponding to the simple modules associated to these weights are
the generalized Baxter $Q$-operators introduced in \cite{FH4} (up to a
normalization). In the following subsections we will recall the
$TQ$-relations which we proved in \cite{FH} and their extended
versions which we conjectured in \cite{FH4}. We will also formulate
a new dual version of these relations for the category $\mathcal{O}^*$
(see Section \ref{nver}).

Recall the monomial $Y_{w(\omega_i),1}$ introduced in Section
\ref{secextm}. For $i\in I$ and $a\in\mathbb{C}^\times$,
  we defined the following monomials in \cite{FH4}:
\begin{align}\label{wia} W_{i,a} = \begin{cases} Y_{i,a}&\text{ if $d_i = d$,}
\\ Y_{i,aq^{-1}}Y_{i,aq}&\text{ if $ d_i = d - 1$,} \\Y_{i,aq^{-2}}Y_{i,a}Y_{i,aq^2}&\text{ if $d_i = d - 2$.} \end{cases}\end{align}
As established in \cite{FH4}, if
  $C_{k,i} < - 1$, then $Y_{w(\omega_i),1}$ is a Laurent polynomial in
  $W_{k,b}$.

\begin{defi}\label{subs}\cite{FH4}
For $i\in I$, $a\in\mathbb{C}^\times$, and $w\in W$, define an
$\ell$-weight $\Psib_{w(\omega_i),a}$ by the following formulas.

The $\ell$-weight $\Psib_{w(\omega_i),1}$ is defined from the
factorization of $Y_{w(\omega_i),1}$ as a product of the variables 
$Y_{k,b}^{\pm 1}$ by replacing

$Y_{k,b}$ by $\Psib_{k,b^{-1}}$ if $d_i = d_k$,

$Y_{k,b}$ by $\Psib_{k,b^{-1}q}\Psib_{k,b^{-1}q^{-1}}$ if $C_{i,k} = -2$,

$Y_{k,b}$ by
  $\Psib_{k,b^{-1}q^2}\Psib_{k,b^{-1}}\Psib_{k,b^{-1}q^{-2}}$ if
$C_{i,k} = -3$,

$W_{k,b}$ by $\Psib_{k,b^{-1}}$ if $C_{k,i} < - 1$.

\noindent Finally, we define $\Psib_{w(\omega_i),a}$ from $\Psib_{w(\omega_i),1}$
by the change of variable $z\mapsto za$.
\end{defi}

For example, we have
$$Y_{s_i(\omega_i),1} = Y_{i,a}A_{i,aq_i}^{-1}\text{ and
}\Psib_{s_i(\omega_i),a} = \wt{\Psib}_{i,aq_i^{-2}}.$$

Recall the operators $T'_i$ on $\Yim'$ introduced in Section
\ref{genchari}.

\begin{prop}\cite{FH4}
We have
$$\Psib_{w(\omega_i),a} = \sigma \circ T'_w \circ
\sigma(\Psib_{i,a}), \qquad w \in W,
$$
where $T'_w = T'_{i_1} T'_{i_2} \ldots T'_{i_k}$
for any reduced decomposition $w = s_{i_1} \ldots s_{i_k}$ and
$\sigma$ is given by formula \eqref{sigm}.
\end{prop}

 Recall that we have the bar involution $i \mapsto \ol{i}$
  on $I$ satisfying $\omega_{\ol{i}} = - w_0(\omega_i)$.

\begin{prop}\label{stabsig} The family
  $\{\Psib_{w(\omega_i),a}\}_{i\in I, w\in W,
  a\in\mathbb{C}^\times}$ is preserved by $\sigma$. More
  precisely, we have
\begin{equation}    \label{precise}
  \sigma(\Psib_{w(\omega_i),a}) = \Psib_{w
    w_0(\omega_{\overline{i}}),a^{-1}q^{r^\vee h^\vee}}.
\end{equation}
	\end{prop}

\begin{proof} Consider the set of all extremal monomials
  $Y_{w(\omega_i),a}$ occurring in the $q$-characters of
all fundamental representations $L(Y_{i,a})$, $i\in I$,
$a\in\mathbb{C}^\times$. According to the results of \cite{h3},
this set is preserved by the automorphism sending $Y_{i,a}\mapsto
Y_{i,a^{-1}}^{-1}$, $i\in I$, $a\in\mathbb{C}^\times$ 
(note that this automorphism was denoted by $\sigma$ in \cite{h3}, but
this is not the automorphism $\sigma$ that 
we use in the present paper for the $Y_{i,a}$ variables). This
automorphism has the same effect on the variables $Y_{i,a}$ as our $\sigma$ has 
on the variables $\Psib_{i,a}$. Since the
    substitutions in Definition
\ref{subs} are compatible with $\sigma$, we obtain the result. 

To prove formula \eqref{precise}, we use the result established in
\cite{h3} that the image of the set of
all monomials occurring in $\chi_q(L(Y_{i,a}))$ under the above
automorphism is the set of monomials occurring
in $\chi_q(L(Y_{\overline{i},a^{-1}q^{-r^\vee h^\vee}}))$. Since $Y_{w(\omega_i),a}$ is the 
unique monomial of weight $w(\omega_i)$ in $\chi_q(L(Y_{i,a}))$ and $w(\omega_i) = -w w_0(\omega_{\overline{i}})$, 
this implies that under this automorphism, we have $Y_{w(\omega_i),a}
\mapsto Y_{w w_0(\omega_{\overline{i}}),a^{-1}q^{-r^\vee h^\vee}}$.
We then obtain formula \eqref{precise} by using Definition \ref{subs}.
\end{proof}

\begin{example} For $\g=sl_2$, we have $\sigma(\Psib_{\omega_1,a}) = \sigma(\Psib_{1,a}) = \Psib_{1,a^{-1}}^{-1} = \Psib_{-\omega_1,a^{-1}q^2}$ and $\sigma(\Psib_{-\omega_1,a}) = \sigma(\Psib_{1,aq^{-2}}^{-1}) = \Psib_{1,a^{-1}q^2} = \Psib_{\omega_1,a^{-1}q^2}$.
\end{example}

Next, we define a ring automorphism $\Omega$ of $\Yim'$ by the formula
\begin{equation}   
\Omega(\Psib_{i,a}) := \Psib_{i,a^{-1}}, \qquad i\in I, \quad
a\in\mathbb{C}^\times; \qquad \Omega([\omega]) = [\omega], \quad
\omega \in P.
\end{equation}

Since $\sigma$ is the composition of $\Omega$ and taking the inverse,
the next statement follows from Proposition \ref{stabsig}.

\begin{prop}\label{imOm} The image of the family $\{\Psib_{w(\omega_i),a}\}_{i\in I,
  a\in\mathbb{C}^\times, w\in W}$ under $\Omega$ is the family
  $\{ \Psib_{w(\omega_i),a}^{-1}\}_{i\in I,
  a\in\mathbb{C}^\times, w\in W}$. More precisely,
\begin{equation}    \label{precise1}
  \Omega(\Psib_{w(\omega_i),a}) = \Psib_{w
    w_0(\omega_{\overline{i}}),a^{-1}q^{r^\vee h^\vee} }^{-1}.
\end{equation}
\end{prop}

\begin{example}\label{exom} For $\g=sl_2$, we have $\Omega(\Psib_{\omega_1,a}) = \Omega(\Psib_{1,a}) = \Psib_{1,a^{-1}} = \Psib_{-\omega_1,a^{-1}q^2}^{-1}$ and $\Omega(\Psib_{-\omega_1,a}) = \Omega(\Psib_{1,aq^{-2}}^{-1}) = \Psib_{1,a^{-1}q^2}^{-1} = \Psib_{\omega_1,a^{-1}q^2}^{-1}$.
		\end{example}

\subsection{Extended $TQ$-relations}    \label{exten}

The following is the {\em Extended $TQ$-relations Conjecture}
 (in $K_0(\OO)$) which we formulated in \cite{FH4}.

\begin{conj}    \label{exttq}
  Let $w\in W$ and let $V$ be a finite-dimensional simple
  $U_q(\wh\g)$-module. Replace every variable $Y_{i,a}, i
  \in I$, appearing in the $q$-character $\chi_q(V)$ with
$$Y_{i,a} \mapsto [w(\omega_i)]
  \dfrac{[L(\Psib_{w(\omega_i),aq_i^{-1}})]}{[L(\Psib_{w(\omega_i),aq_i})]}.
  $$
 By equating the resulting expression with $[V]$ and
  clearing the denominators, we obtain an algebraic
  relation in $K_0(\OO)$.
\end{conj}

The generalized Baxter $Q$-operators $Q_{w(\omega_i),a}, i \in I, w
\in W$, were defined in \cite{FH4} as the transfer-matrices
corresponding to the simple modules $L(\Psib_{w(\omega_i),a})$ in the
category $\OO$, divided by their ordinary characters. Hence they (and
their eigenvalues) should satisfy the same relations as the classes of
these simple modules in $K_0(\OO)$ -- these are the relations stated
in the above conjecture. This explains the term ``$TQ$-relations''
(see the Introduction of \cite{FH} for more details on the origins of
these relations).
  
For $w=w_0$, we proved Conjecture \ref{exttq} in \cite{FH}.
In the case
when $w$ is a simple reflection, $w=s_i, i \in I$, we proved this
conjecture in \cite{FH2, FHR}.

  These results motivated us to define in
  \cite{FH3} the Weyl group action on an extension of
  $\Yim$, such that the subring of its invariants in $\Yim$ is equal
  to the image of the $q$-character homomorphism $\chi_q$ (for more
  details, see \cite[Section 3.5]{FH4}).

\subsection{Dual extended $TQ$-relations}\label{nver} 

Analogous relations can be written for the
classes of simple modules in the category
$\mathcal{O}^*$. We call
them the {\em Dual Extended $TQ$-relations} (in
  $K_0(\OO^*)$).

\begin{conj}    \label{exttqd}
  Let $w\in W$ and let $V$ be a finite-dimensional simple
  $U_q(\wh\g)$-module. Replace every variable $Y_{i,a}, i
  \in I$, appearing in the $q$-character $\chi_q(V)$ with 
  $$Y_{i,a} \mapsto [-w(\omega_i)]
  \dfrac{[L'(\Psib_{w(\omega_i),aq_i^{-1}}^{-1})]}{[L'(\Psib_{w(\omega_i),aq_i}^{-1})]}.
  $$
By equating the resulting expression with
    $[\tau_{q^{-2h^\vee r^\vee}}(V^*)]$
    and clearing the denominators, we obtain an algebraic relation in
    $K_0(\OO^*)$.	
	\end{conj}

\begin{prop} Conjecture \ref{exttqd} is equivalent to Conjecture
  \ref{exttq}.
\end{prop}

\begin{proof} Recall from Proposition \ref{dualweight} that
  $L'(\Psib_{w(\omega_i),a}^{-1})^* \simeq L(\Psib_{w(\omega_i),a})$
  and from Remark \ref{tau} that $(\tau_{q^{-2h^\vee r^\vee}}(V^*))^*
  \simeq V$. We also have $[\lambda]^* \simeq
  [-\lambda]$. Therefore the relations of Conjecture
  \ref{exttq} are obtained by dualizing the relations of Conjecture
  \ref{exttqd}.
\end{proof}

\begin{example} Let $\g=sl_2$ and $V = L(Y_{1,bq^{2}})$ with $\chi_q(Y_{1,bq^{2}}) = Y_{1,bq^{2}} + Y_{1,bq^{4}}^{-1}$. Then $V^* = L(Y_{1,bq^4})$ and using Example \ref{exom}, we can compute the relation.

For $w = e$, the substitution is $Y_{1,a} \mapsto
[-\omega_1][L'(\Psib_{1,aq^{-1}}^{-1})]
/[L'(\Psib_{1,aq}^{-1})]$. Then
$$[L(Y_{1,b})] [L'(\Psib_{1,bq^3}^{-1})] = [-\omega_1] [L'(\Psib_{1,bq}^{-1})]  +  [\omega_1] [L'(\Psib_{1,bq^5}^{-1})].$$
For $w = s_1$,  the substitution is $Y_{1,a} \mapsto [\omega_1][L'(\Psib_{1,aq^{-3}})]
/[L'(\Psib_{1,aq^{-1}})]$. Then
$$[L(Y_{1,b})] [L'(\Psib_{1,bq})]= [\omega_1][L'(\Psib_{1,bq^{-1}})] + [-\omega_1][L'(\Psib_{1,bq^3})].$$
Thus, we recover the examples discussed in \cite[Section 4.3]{FH}.
\end{example}

 In order to relate the Dual Extended $TQ$-relations of
  Conjecture \ref{exttqd} to
  the polynomiality of the generalized Baxter operators discussed in
  the next subsection, we need to
  rewrite these relations in terms of the modules of the form
  $L'(\Omega(\Psib_{ww_0(\omega_j),b}), j \in I, b \in \C^\times$.
  
  \begin{prop}    \label{Omegarel}
    Conjecture
  \ref{exttqd} is equivalent to the following statement: Replace every
  variable $Y_{i,a}, i \in I$, appearing in
$\chi_q(V)$ with 
$$Y_{i,a} \mapsto [-w(\omega_i)]
  \dfrac{[L'(\Omega(\Psib_{ww_0(\omega_{\overline{i}}),a^{-1}q^{r^\vee h^\vee}q_i}))]}{[L'(\Omega(\Psib_{ww_0(\omega_{\overline{i}}),a^{-1}q^{r^\vee h^\vee}q_i^{-1}}))]}.
  $$
   By equating the resulting expression with
    $[\tau_{q^{-2h^\vee r^\vee}}(V^*)]$
    and clearing the denominators, we obtain an algebraic relation in
    $K_0(\OO^*)$.
    \end{prop}

\begin{proof}
 Applying formula \eqref{precise1}, we obtain the
  statement of the proposition.
\end{proof}

For $w = w_0$, we proved the relations of Conjecture \ref{exttqd}
in \cite[Theorem 4.8]{FH}. The corresponding family of
    modules is
    $$\{ L'(\Omega(\Psib_{\omega_{\overline{i}},a^{-1}})) \}_{i \in I, a
        \in \C^\times} 
      = \{ L'(\Psib_{\omega_j,b})
      \}_{j \in I, b \in \C^\times}$$		
       (see
      \cite[Example 3.8]{FH4}), which is the family
      $\{  R_{j,b}^+ \}_{j \in I, b \in \C^\times}$ we discussed
      above.

      As explained in \cite{FH}, the relations of Conjecture
      \ref{exttqd} for $w=w_0$ imply the same
        relations on the joint eigenvalues (or spectra) of the
        corresponding transfer-matrices. This enables us to express
        the spectra of the transfer-matrices $t_V(z,u)$, where $V$ is
        a simple finite-dimensional $U_q(\ghat)$-module (these are the
        commuting {\em Hamiltonians} of the quantum integrable model
        of XXZ type associated to $U_q(\ghat)$), in terms of the
        spectra of the transfer-matrices corresponding to $R_{j,b}^+,
        j \in I, b \in \C^\times$. In \cite{FH}, we proved that
        the spectra of these transfer-matrices on the tensor products of simple
        finite-dimensional $U_q(\ghat)$-modules are in fact {\em
          polynomials}, up to an overall factor. Therefore, we obtain
        a description of the spectra of the Hamiltonians of the
        quantum integrable model of XXZ type associated to
        $U_q(\ghat)$ in terms of these polynomials. As explained in
        \cite{FH}, this description is closely related to the
        description in terms of the solutions of the corresponding
        Bethe Ansatz equations, but is in fact more
        direct. (In the case of $\g=sl_2$, this goes back to the results
        of Baxter on the polynomiality of his $Q$-operator
        \cite{Baxter}.)

  The discussion of the preceding paragraph concerns the
  $TQ$-relations in the case $w=w_0$ and their implications for the
  spectra of the XXZ type model associated to $U_q(\ghat)$. And now
  we want to generalize this to the case of an arbitrary $w \in W$.
  
  Namely, we can view the family of modules
			       $\{
  L'(\Omega(\Psib_{ww_0(\omega_j),b})) \}_{w \in W, j \in I, b \in \C^\times}$
				as an extension of the family $\{
        R_{j,b}^+ \}_{j \in I, b \in \C^\times}$ (corresponding to
        $w=w_0$) to all $w \in
        W$. As in the case $w=w_0$, given any $w \in W$, the
        corresponding $TQ$-relations, in the
        form stated in
        Proposition \ref{Omegarel}, imply the same relations on
        the spectra of the corresponding
        transfer-matrices. This enables us to express the spectra of the
  transfer-matrices $t_V(z,u)$, where $V$ is a simple finite-dimensional
  $U_q(\ghat)$-module, in terms of the spectra of the
  transfer-matrices corresponding to the modules
  $L'(\Omega(\Psib_{ww_0(\omega_j),b}))$ with a fixed $w \in W$.
        
        In the next subsection we
        will discuss the polynomiality of the spectra
        of the transfer-matrices
        of the modules
        $L'(\Omega(\Psib_{w(\omega_j),b})), w \in W, j \in
        I$. On the one hand, this gives us an alternative description
        of the spectra of the Hamiltonians of the quantum
  integrable model of XXZ type associated to $U_q(\ghat)$ in terms of
  these polynomials, for each $w \in W$. On the other hand, we will
  show (Proposition \ref{limitt}) that in a certain
  limit, the transfer-matrix of the module
        $L'(\Omega(\Psib_{w(\omega_j),b}))$ becomes the $X$-series
  $X_{w(\omega_i)}(zb)$. Thus, polynomiality of the spectra of these
  transfer-matrices is closely related to the polynomiality of the
  spectra of these $X$-series, which is one of the main themes of the
  present paper.

\subsection{Polynomiality of the generalized Baxter operators}\label{tmconj}

 Denote by $t_{w(\omega_i)}(z,u)$ the transfer-matrix of the
  simple module $L'(\Omega(\Psib_{w(\omega_i),1}))$ in the category
  $\mathcal{O}^*$. By Proposition \ref{imOm}, these are the
  $\OO^*$-versions of the generalized Baxter operators 
    (which are the
  transfer-matrices
corresponding to the simple modules $L(\Psib_{w(\omega_i),a})$ in the
category $\OO$, see Section \ref{exten}).
  By construction, the ordinary weights of
  $L'(\Psib_{w(\omega_i),1})$ belong to $Q^+ \subset Q$. Therefore,
  according to the results of
  Section \ref{Rgeq}, $t_{w(\omega_i)}(z,u)
  \in U(\wh\bb_-)[[z,v_i]]_{i\in I}$ and hence it has a
  well-defined limit at $v = 0$.

  The following result generalizes Theorem \ref{limit}.

\begin{prop}\label{limitt} The limit of the transfer-matrix
  $t_{w(\omega_i)}(z,u)$ at $v = 0$ is equal to the
  $X$-series $X_{w(\omega_i)}(z)$.
\end{prop}

\begin{proof} We use the same argument as in the proof of Theorem
  \ref{limit} (which is \cite[Theorem 5.5]{FH} proved in \cite[Section
    7.1]{FH}). Namely, the transfer-matrix $t_{w(\omega_i)}(z)$ is a
  formal Taylor power series in the variables $v_j$ with coefficients
  in $U_q(\wh{\mathfrak{g}})[[z]]$. At the limit when all $v_j$ tend
  to $0$, we obtain the constant term which is the trace on the lowest
  weight vector of $L'(\Omega(\Psib_{w(\omega_i),1}))$
  (of weight 0).

  For weight reasons, only the abelian factor $\mathcal{R}^0$ can
  contribute to this trace (see \cite[Section 7.1]{FH}), and so the
  limit depends only on the eigenvalues of the Drinfeld-Cartan
  generators on the lowest weight vector of
  $L'(\Omega(\Psib_{w(\omega_i),1}))$. To compute the latter, we use
  the argument similar to the one
  in \cite[Section 5.3]{FH}.  Namely, comparing
    Proposition \ref{factW} and
  Definition \ref{subs}, we find that the factorization of the lowest
  $\ell$-weight  $\Omega(\Psib_{w(\omega_i),1})$ in terms of the
	$\Psib_{j,b}$'s is the same as the factorization of
  $X_{w(\omega_i)}(z)$ in terms of the
	$X_{\omega_j}(zb)$'s.
\end{proof}

Note that in general $L'(\Omega(\Psib_{w(\omega_i),1}))$ is {\em not}
isomorphic to the tensor product of prefundamental representations
\begin{equation}    \label{M}
M = \bigotimes_{j\in I, b\in\mathbb{C}^\times}
L'(\Psib_{j,b}^{m_{j,b}^i}),
\end{equation}
where we write $\Omega(\Psib_{w(\omega_i),1}) = 
  \prod_{j\in I, b\in\mathbb{C}^\times}
\Psib_{j,b}^{m_{j,b}^i}$. Indeed, $L'(\Omega(\Psib_{w(\omega_i),1}))$ is
a submodule of the tensor product $M$, but $M$ is {\em not}
necessarily simple. In this case, the transfer matrices $t_M(z,u)$
and $\prod_{j\in I, b\in\mathbb{C}^\times}
t_{\omega_j,b}(z,u)^{m_{j,b}^i}$ are not equal to each other. However, in
the limit when the $v_j$'s
tend to $0$, they become equal as the contributions of all additional
simple constituents in the tensor product\eqref{M} drop out.

Therefore, Proposition \ref{limitt} implies the following result.

\begin{prop} The limit of the transfer-matrix $t_{M}(z,u)$ at $v=0$
  is equal to $X_{w(\omega_i)}(z)$.
\end{prop}

This motivates the following conjecture. Let $L(m)$ be a simple
finite-dimensional $U_q(\wh{\mathfrak{g}})$-module.

\begin{conj}    \label{trm}
  The operator $(f_{i,m}(z))^{-1}t_{w(\omega_i)}(z,u)$ acting on any
  simple finite-dimensional $U_q(\ghat)$-module $L(m)$ is a polynomial
  in $z$.
\end{conj}

  \begin{thm}\cite{FH} Conjecture \ref{trm} holds if $w$ is the
    identity element.
\end{thm}
    
\begin{rem} An analogue of this statement, for a simple module in the category $\mathcal{O}$ corresponding to the
  longest element of the Weyl group, was proved in \cite{Z}. It also
  implies Conjecture \ref{twpol} for the longest element of the Weyl
  group because the limit of the corresponding transfer-matrix
  is the $X$-series associated to this element.
\end{rem}

\begin{rem}    \label{last}
Using the dual extended $TQ$-relations from Section \ref{nver}, we now
obtain a conjectural description, for each fixed $w \in W$, of the
spectra of the transfer-matrices $t_V(z,u)$, where $V$ are simple
finite-dimensional $U_q(\ghat)$-modules in terms of polynomial
eigenvalues of the renormalized generalized Baxter operators
$(f_{i,m}(z))^{-1} t_{w(\omega_i)}(z,u)$. This
generalizes the description of the spectra of the $t_V(z,u)$'s in
terms of the polynomial eigenvalues of $ (f_{i,m}(z))^{-1}
t_{R_{i,1}^+}(z,u), i \in I$, proved in \cite{FH}, which corresponds
to the case $w=w_0$ (we note that it had been conjectured in
\cite{Fre}). Thus, we obtain many different descriptions, labeled by
elements
$w \in W$, of the spectra of the Hamiltonians (the transfer-matrices
$t_V(z,u)$) of the corresponding quantum integrable model of XXZ
type.
\end{rem}

\medskip

\noindent The manuscript has no associated data.

\end{document}